\documentclass[reqno,10pt]{amsart}

\usepackage[utf8]{inputenc}
\usepackage[T2A]{fontenc}
\usepackage[russian, main=english]{babel}
\usepackage{enumitem}

\usepackage{amsmath,amsthm,amsfonts, bbm, amssymb, hyperref,graphicx,color}
\usepackage{caption, subcaption}
\def\[#1\]{\begin{align}#1\end{align}}
\def\(#1\){\begin{align*}#1\end{align*}}

\newcommand{\ii}{\mathbbm{i}}

\renewcommand{\C}{\mathbb{C}}

\newcommand{\R}{\mathbb{R}}
\newcommand{\Z}{\mathbb{Z}}

\newcommand{\kk}{\mathbbm{k}}
\newcommand{\jj}{\mathbbm{j}}

\newcommand{\norm}[1]{\left\vert #1 \right \vert}	
\newcommand{\Norm}[1]{\left\Vert #1 \right \Vert}

\renewcommand{\ker}{\text{Ker }}

\newcommand{\Isom}{\text{Isom}}

\newcommand{\rad}{\operatorname{rad}}

\newcommand{\comment}[1]{}

\newcommand{\ignore}[1]{}

\newtheorem{thm}{Theorem}

\newtheorem{prop}[thm]{Proposition}
\newtheorem{lemma}[thm]{Lemma}
\newtheorem{cor}[thm]{Corollary}

\numberwithin{thm}{section}

\theoremstyle{definition}
\newtheorem{defi}[thm]{Definition}

\theoremstyle{definition}
\newtheorem{example}[thm]{Example}
\newtheorem{remark}[thm]{Remark}

\numberwithin{equation}{section}

\newcommand{\mst}{\;:\;}

\newcommand{\Zee}{\mathcal Z}

\newcommand{\cupover}{\bigcup}

\newcommand{\Sph}{\mathbb S}

\title[Serendipitous decompositions]{Serendipitous decompositions of higher-dimensional continued fractions}
\author[A. Lukyanenko]{Anton Lukyanenko}
\address{
Department of Mathematics\\
George Mason University\\
4400 University Drive, MS: 3F2\\
Fairfax, Virginia 22030}
\email{alukyane@gmu.edu}
\author[J. Vandehey]{Joseph Vandehey}
\address{
Department of Mathematics\\
University of Texas at Tyler\\
Tyler, TX 75799
}
\email{jvandehey@uttyler.edu}


\begin{document}

\begin{abstract}
    We prove a suite of dynamical results, including exactness of the transformation and piecewise-analyticity of the invariant measure, for a family of continued fraction systems, including specific examples over reals, complex numbers, quaternions, octonions, and in $\R^3$. Our methods expand on the work of Nakada and Hensley, and in particular fill some gaps in Hensley's analysis of Hurwitz complex continued fractions. We further introduce a new ``serendipity'' condition for a continued fraction algorithm, which controls the long-term behavior of the boundary of the fundamental domain under iteration of the continued fraction map, and which is under reasonable conditions equivalent to the finite range property. We also show that the finite range condition is extremely delicate: perturbations of serendipitous systems by non-quadratic irrationals do not remain serendipitous, and experimental evidence suggests that serendipity may fail even for some rational perturbations.
\end{abstract}

\maketitle

\section{Introduction}
\label{sec:intro}
The regular continued fraction (CF) expansion of an irrational real number $x_0\in [0,1)$,
\(
x_0= \cfrac{1}{a_1+\cfrac{1}{a_2+\cfrac{1}{a_3+\dots}}}, \quad a_i\in \mathbb{N},
\)
expresses $x_0$ as an alternating sequence of inversions $\iota(x)=1/x$ and shifts $x\mapsto a+x$ for integers $a$. The CF map $T:[0,1)\to [0,1)$ given by $Tx = 1/x - \lfloor 1/x\rfloor$ acts as a forward shift on the sequence of digits. The map $T$, often called the Gauss map, satisfies a number of important dynamical properties. In particular, $T$ is exact\footnote{All dynamical statements are with respect to Lebesgue measure, unless otherwise stated.} (and thus ergodic) and satisfies a Kuzmin-type theorem: non-singular probability distributions converge at an exponential rate to the invariant measure with density $\frac{1}{\log 2}\frac{1}{1+x}$. For more general information, see \cite{DK,EW,HensleyBook,IK}.

In this paper, we provide a unified analysis of a wide range of Iwasawa CFs \cite{lukyanenko_vandehey_2022} in $\R^d$, including certain complex, quaternionic, and octonionic CFs, as well as more exotic systems such as 3D CFs  with digits in $\Z^3$ and inversion $\iota(x,y,z)=\frac{(x,y,z)}{x^2+y^2+z^2}$. In particular, our Main Theorem \ref{thm:main} implies:
\begin{thm}
\label{thm:kindamain}
The CF system associated to the Hurwitz integers within the quaternions (see Example \ref{example:quaternions} in \S \ref{sec:examples}) is exact, continued-fraction-mixing, satisfies a Kuzmin-type theorem, and has a unique invariant measure equivalent to Lebesgue measure, whose density is bounded and piecewise-analytic with finitely many pieces.
\end{thm}

\begin{figure}
    \centering
     \hfill{}
     \begin{subfigure}[b]{\textwidth}
         \centering
        \includegraphics[width=.23\textwidth]{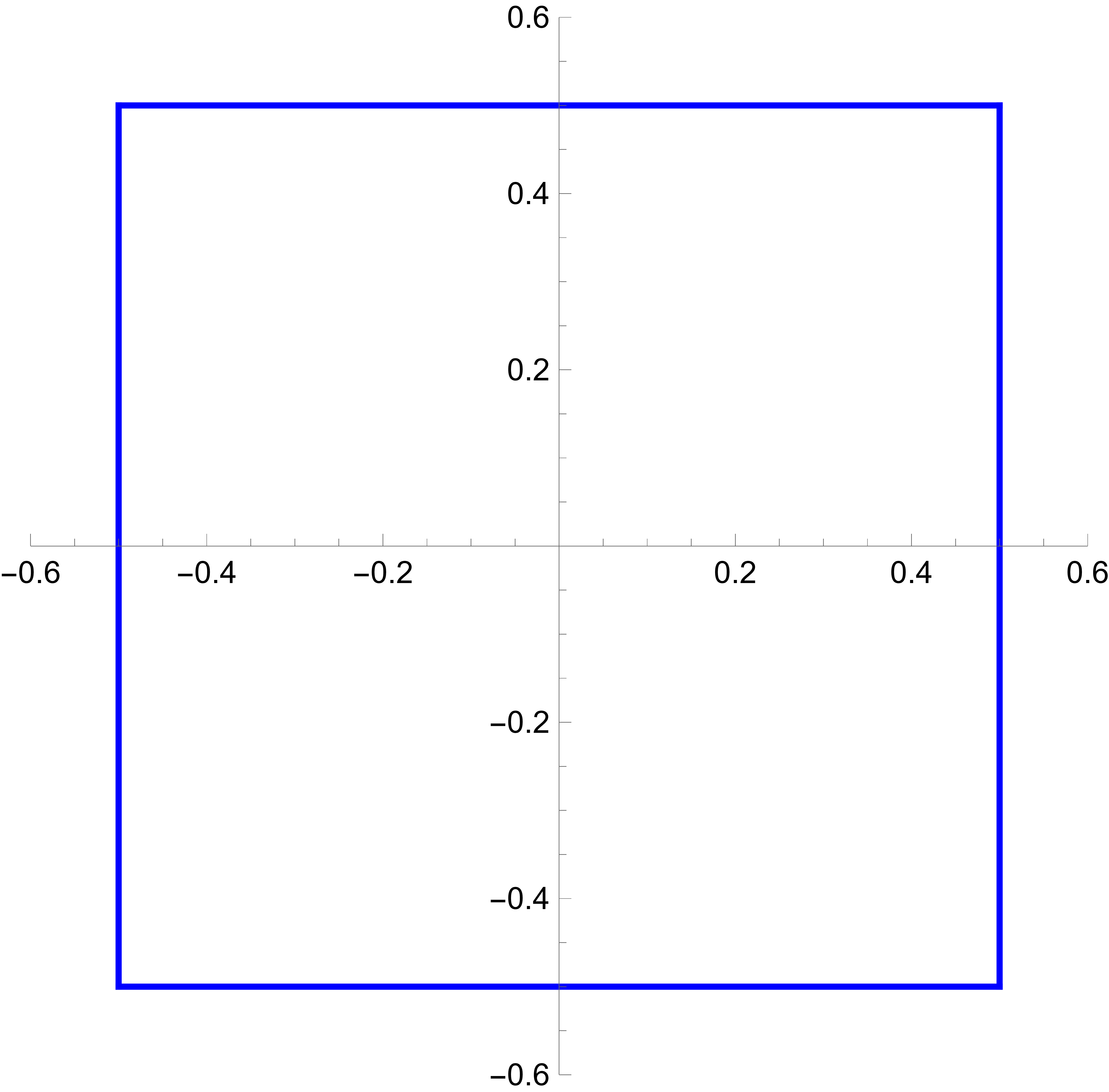}
        \includegraphics[width=.23\textwidth]{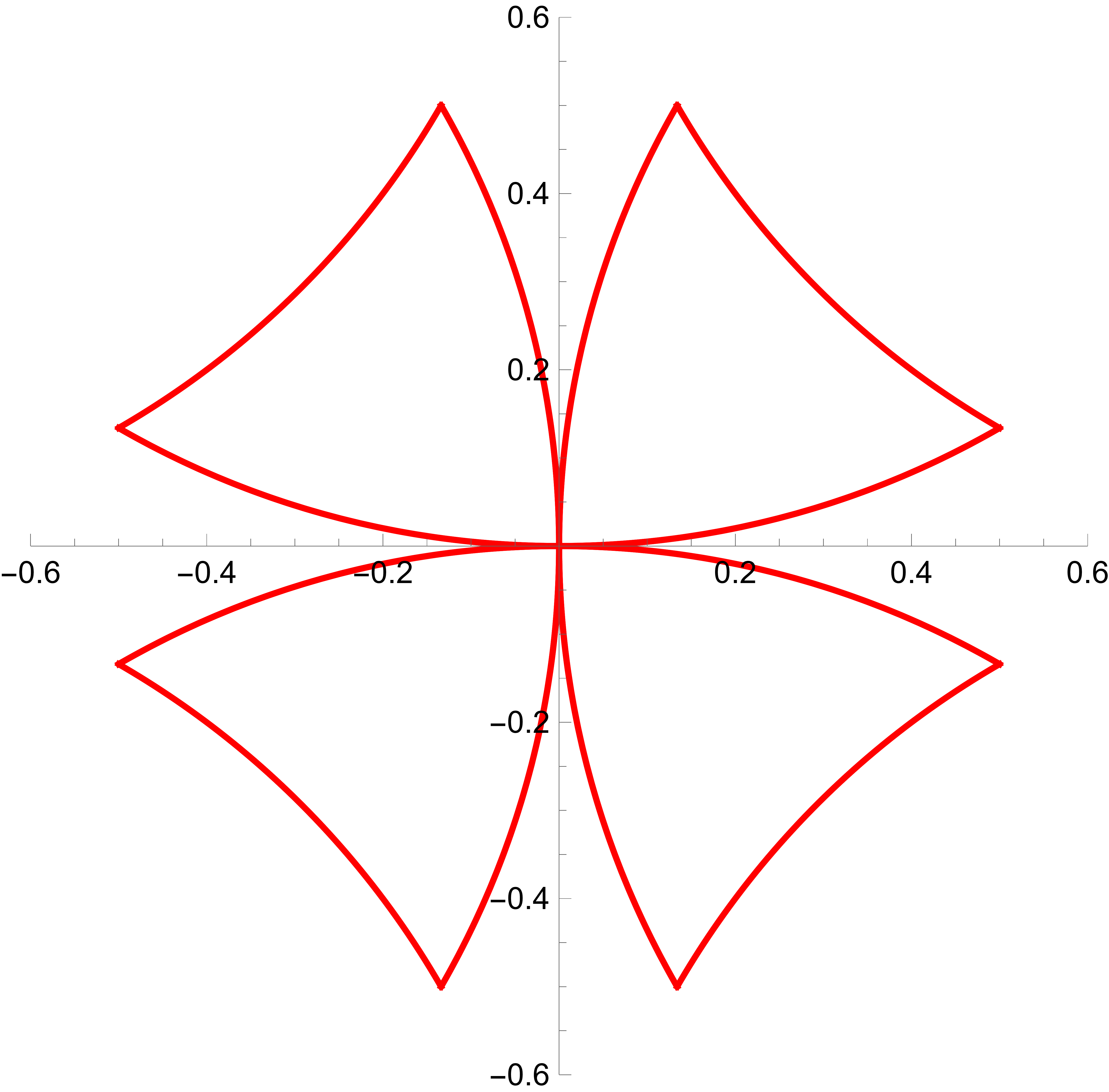}
        \includegraphics[width=.23\textwidth]{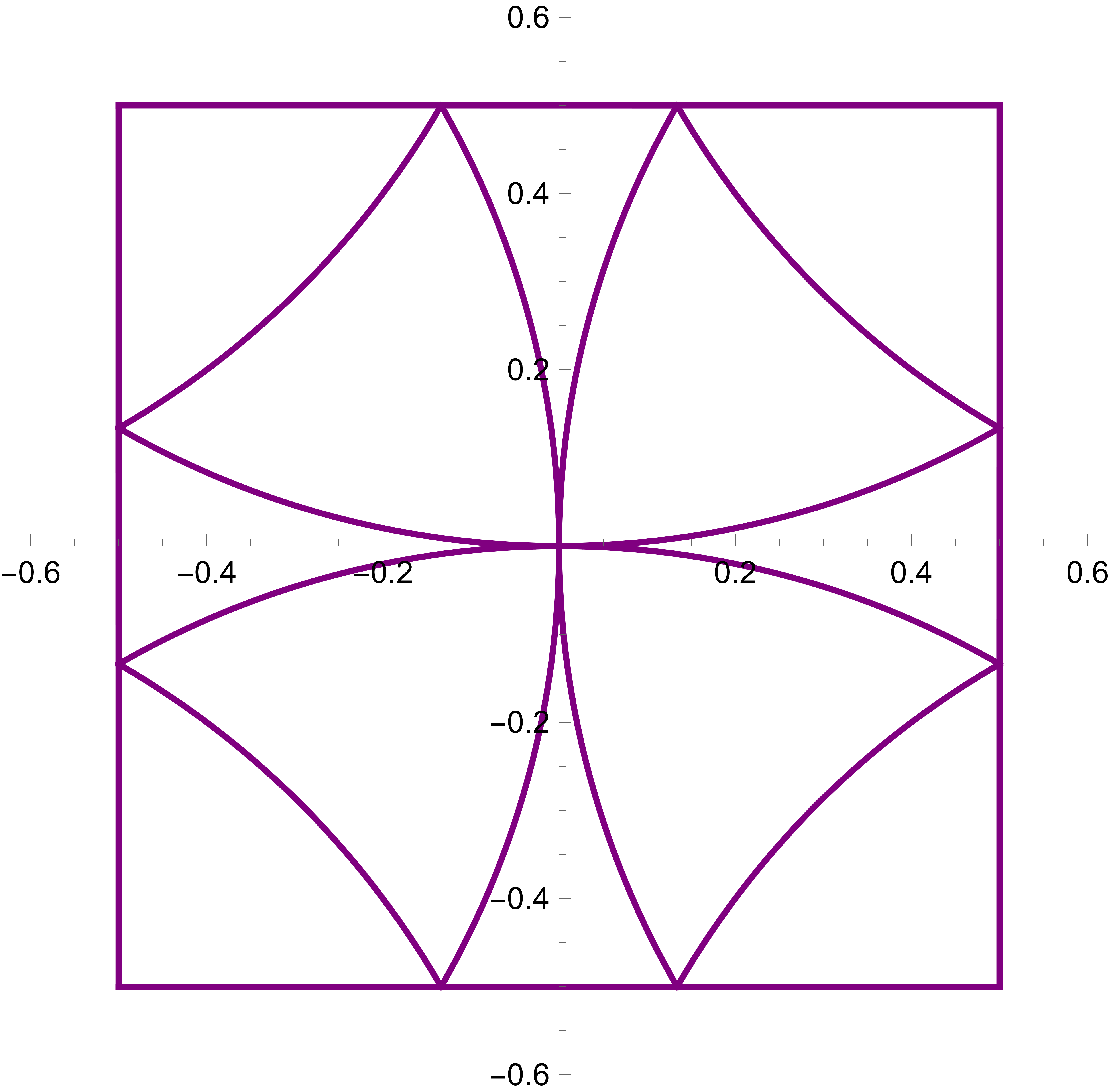}
        \includegraphics[width=.23\textwidth]{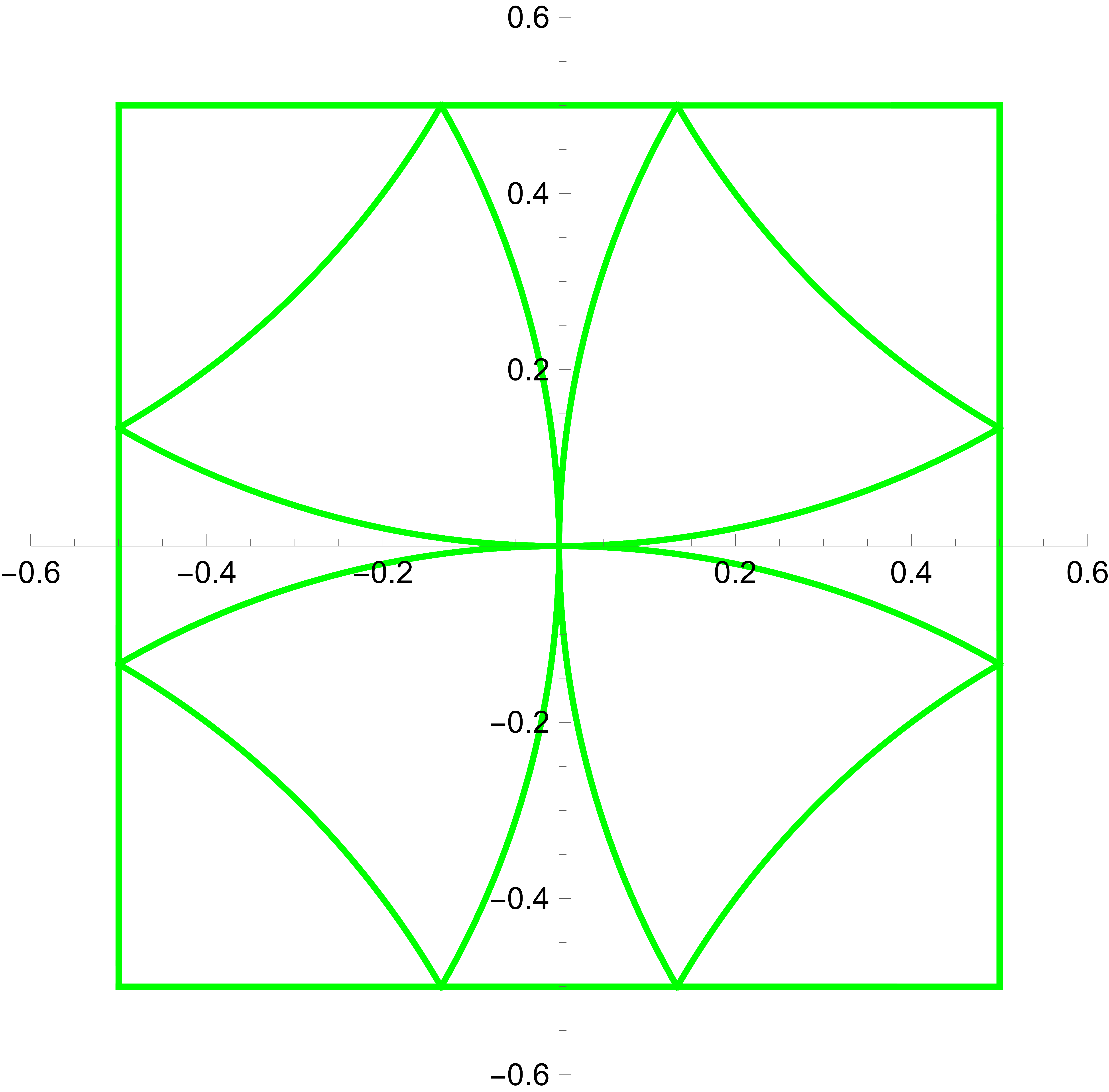}
         \caption{The Hurwitz complex CF, with $K=[-0.5,0.5)\times [-0.5,0.5)$.}\label{fig:planar Hurwitz}
     \end{subfigure}\\
     \begin{subfigure}[b]{\textwidth}
         \centering
        \includegraphics[width=.23\textwidth]{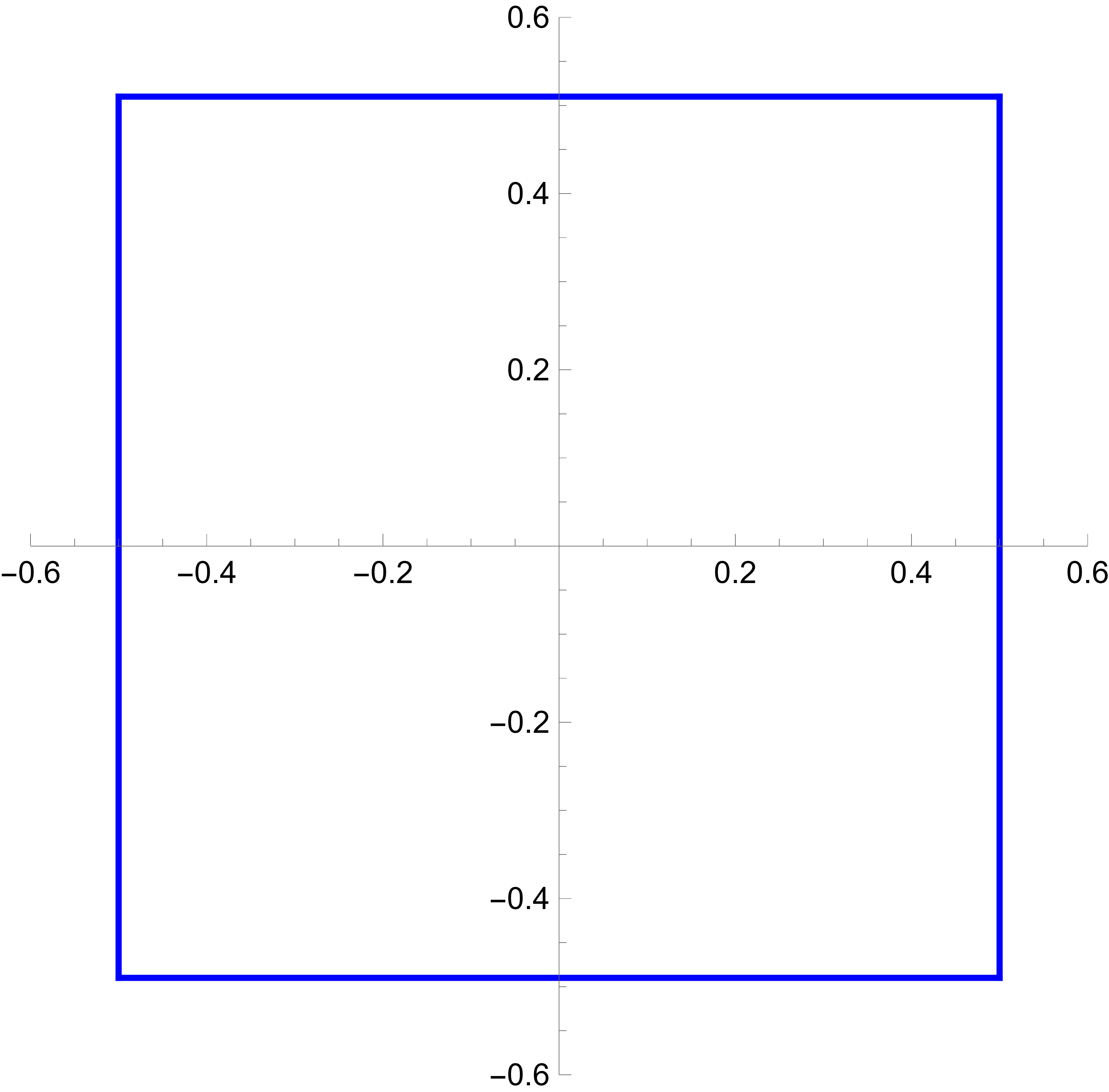}
        \includegraphics[width=.23\textwidth]{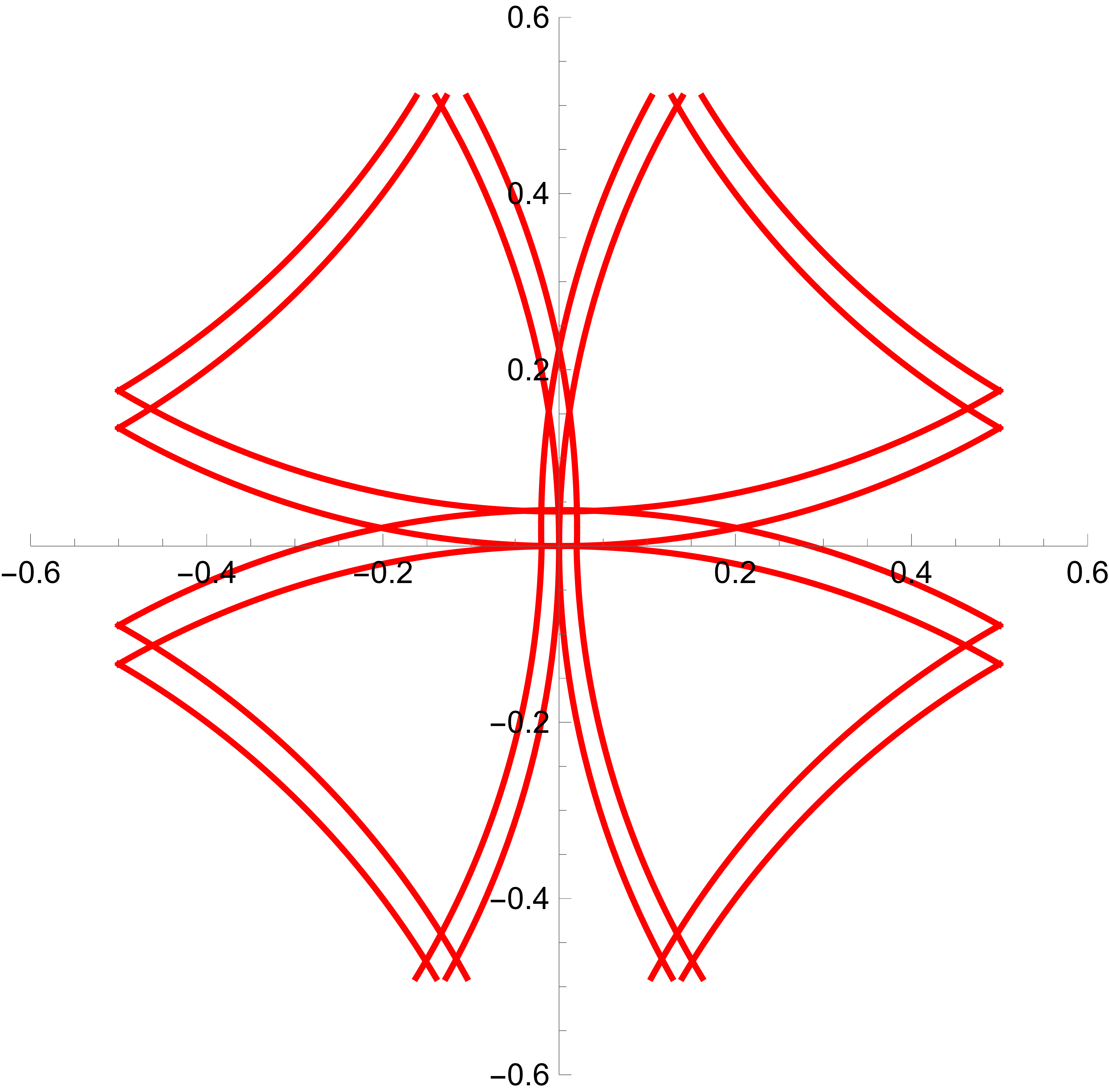}
        \includegraphics[width=.23\textwidth]{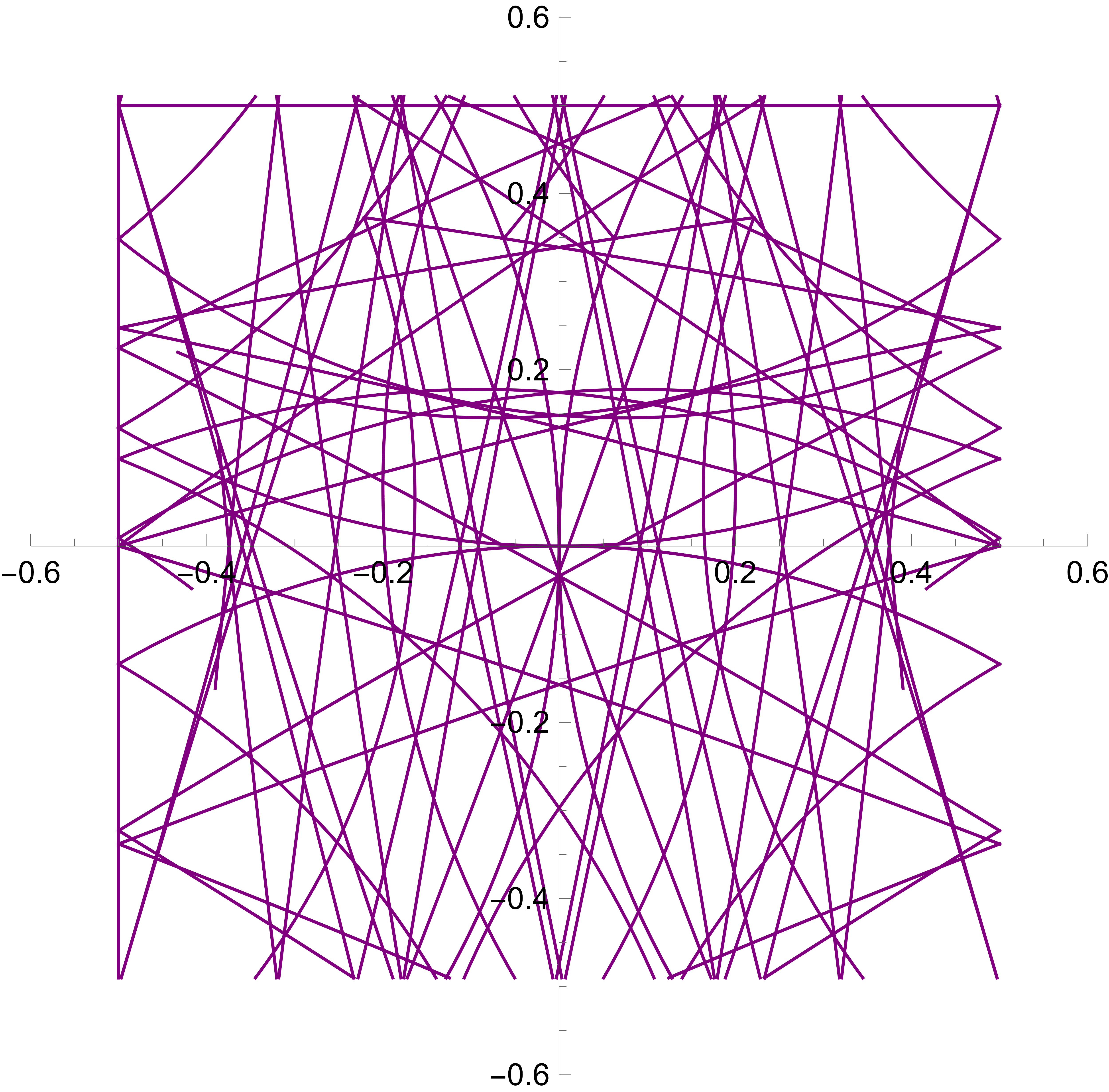}
        \includegraphics[width=.23\textwidth]{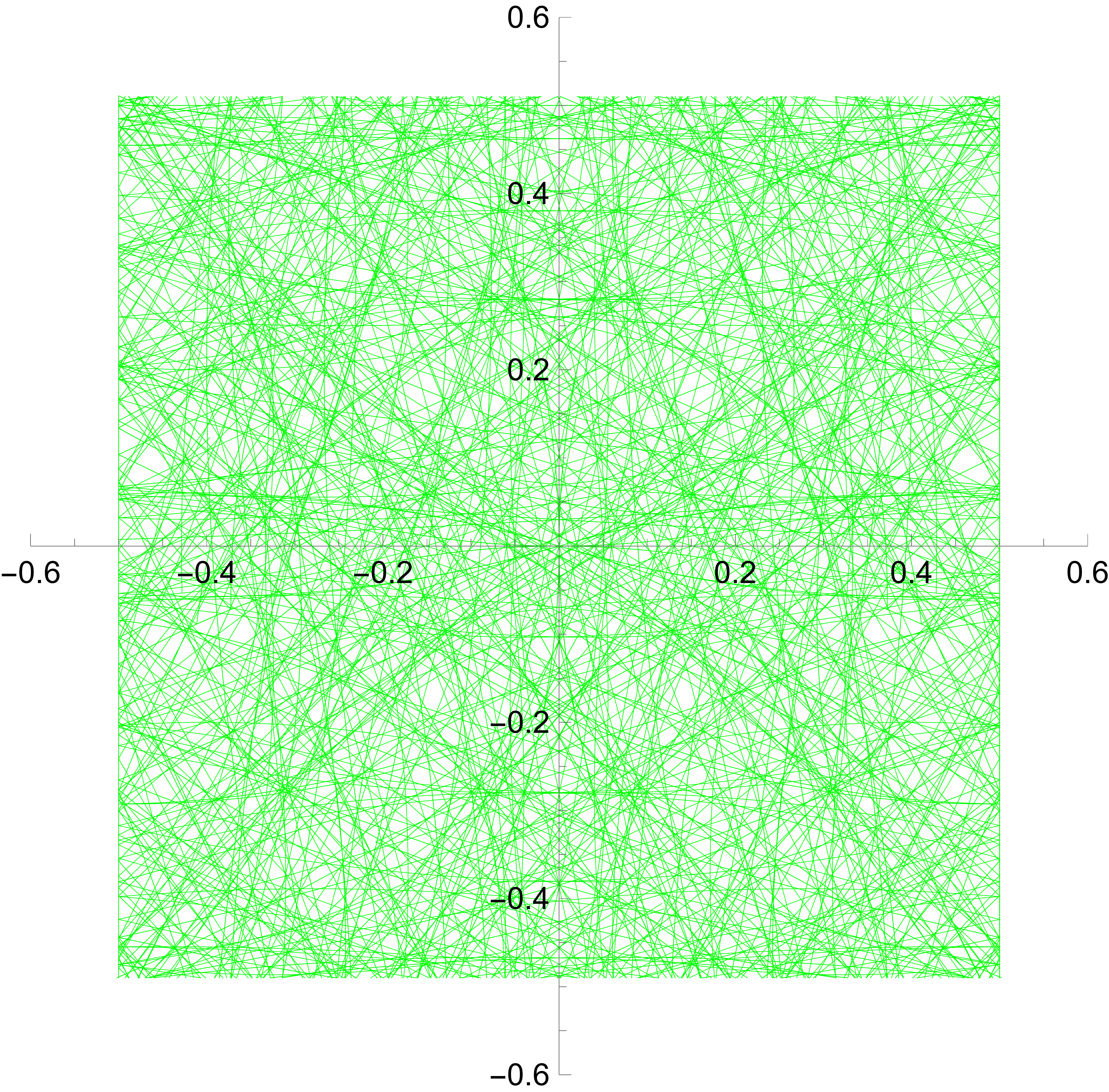}
         \caption{A perturbed Hurwitz CF, with $K=[-0.5,0.5)\times [-0.49,0.51)$.}\label{fig:planar offset}
     \end{subfigure}\\
     \begin{subfigure}[b]{0.7\textwidth}
         \centering
        \includegraphics[width=.8\textwidth]{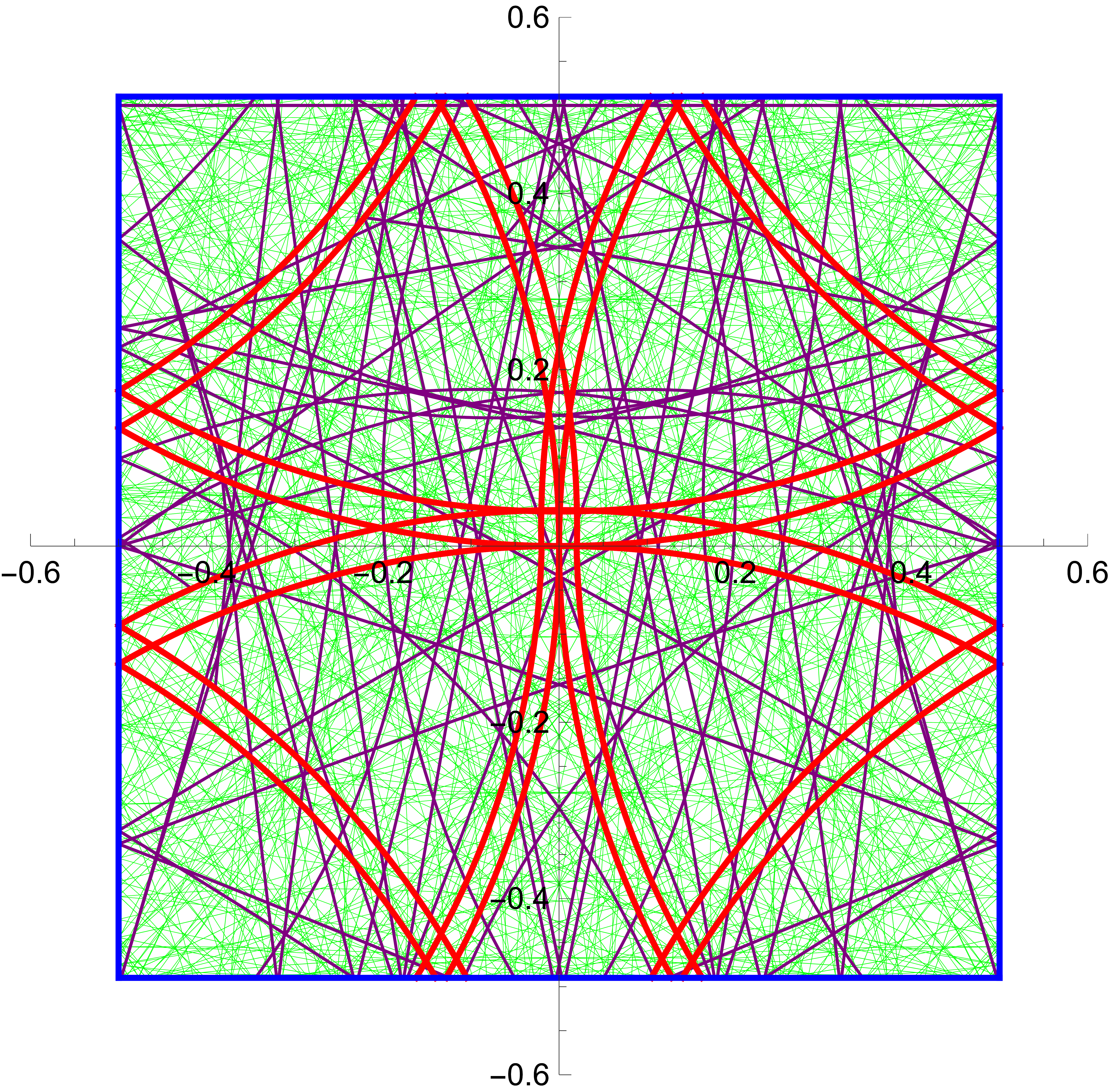}
         \caption{A composite picture of the perturbed Hurwitz CF.}\label{fig:planar composite}
     \end{subfigure}
    \caption{The Hurwitz continued fraction associated to $\Z[i]$ and $\iota(z)=1/z$ has a serendipitous decomposition. The sequence (a) shows, left to right, $T^i \partial K$ for $i=0, \ldots, 3$. Even minor perturbations of the algorithm, such as shifting $K$ up by $0.01\ii$, appear to destroy the decomposition (b, c).}
    \label{fig:planar}
\end{figure}

Our primary tool will be the analysis of cylinder sets, which is commonly employed under a full-cylinder condition (satisfied by regular CFs) as in \cite{bandtlow2007invariant} or a finite-range condition (satisfied by nearest-integer CFs and A.~Hurwitz complex CFs) as in \cite{HensleyBook,Nakada1976}. If $T$ is the CF map on a set $K$, then a (rank-$n$) cylinder set $C_{a_1 a_2\dots a_n}\subset K$ is the set of all numbers in $K$ whose expansion begins with the digits $a_1, a_2,\dots, a_n$. A cylinder is \emph{full} if $T^n C_{a_1\dots a_n}=K$; while the \emph{finite range} condition assumes that there are finitely many possibilities for what  $T^n C_{a_1\dots a_n}$ could be. 

We will work with the finite range condition by rephrasing it as the equivalent (in our setting, see Lemma \ref{lemma:FRisSer}) condition that we call \emph{serendipity}. Namely, we will say that an Iwasawa CF system is \emph{serendipitous} if the image $E=\cupover_{k=0}^\infty T^k \partial K$ of the boundary of the fundamental domain $K$ under the map $T$ stabilizes after \emph{finitely many} iterations, so that $E=\cupover_{k=0}^n T^k \partial K$ for some $n$, and if furthermore $K\setminus E$ has finitely many connected components (giving the \emph{serendipitous decomposition} of $K$). In the one-dimensional case, the serendipity condition reduces to the \emph{finiteness condition} of \cite{katok2010structure}.

Figure  \ref{fig:planar Hurwitz} demonstrates serendipity for the square lattice in $\R^2$ (corresponding to A.~Hurwitz complex CFs), with 12 connected components in $K\setminus E$; any set of the form $T^n C_{a_1\dots a_n}$ is then a union of some of these connected components, up to a measure-zero subset of $E$.  Serendipity for the the cubic lattice in $\R^3$, and the rhombic dodecahedral lattice in $\R^3$, are illustrated in Figure \ref{fig:cube and dodecahedron}. Additional examples fitting the assumptions of Theorem \ref{thm:main}, including the serendipity assumption, are discussed in Section \ref{sec:examples}.

Using serendipity, we can demonstrate that the finite range property is an unstable condition. For $\alpha$-CFs (namely, the space $\R$ with digits $\Z$, fundamental domain $K=(-\alpha, 1-\alpha]$, say with inversion $-1/x$, see Lemma \ref{lemma:realserendipity}) serendipity implies that the boundary points $\alpha$ and $1-\alpha$ have a finite orbit under the $\alpha$-CF mapping $T$. This then implies that $\alpha$ is the root  (rational or irrational) of a quadratic equation. In higher dimensions, embedding $\alpha$-CFs as a subsystem provides some negative results. In particular, the above real $\alpha$-CFs are a subsystem, along the imaginary axis, of the $\alpha$-perturbed A.~Hurwitz CF (with space $\R^2$, digits $\Z^2$, fundamental domain $(-.5,.5]\times(-\alpha, 1-\alpha]$, and $\iota(x,y)=(x,-y)/(x^2+y^2)$). When $\alpha$ is not a quadratic surd, it is easy (Corollary \ref{cor:aphahurwitz}) to conclude that the resulting system is not serendipitous. However, even with $\alpha$ a quadratic surd, serendipity does not follow without further assumptions in higher dimensions, and experimental evidence (Figure \ref{fig:planar offset}) suggests it may fail even for rational perturbations. It likewise remains an open question to find a serendipitous system in the non-Euclidean setting of the Heisenberg group, or prove that one does not exist.

\begin{figure}[ht]
    \centering
     \hfill{}
     \begin{subfigure}[b]{0.45\textwidth}
         \centering
        \includegraphics[width=\textwidth]{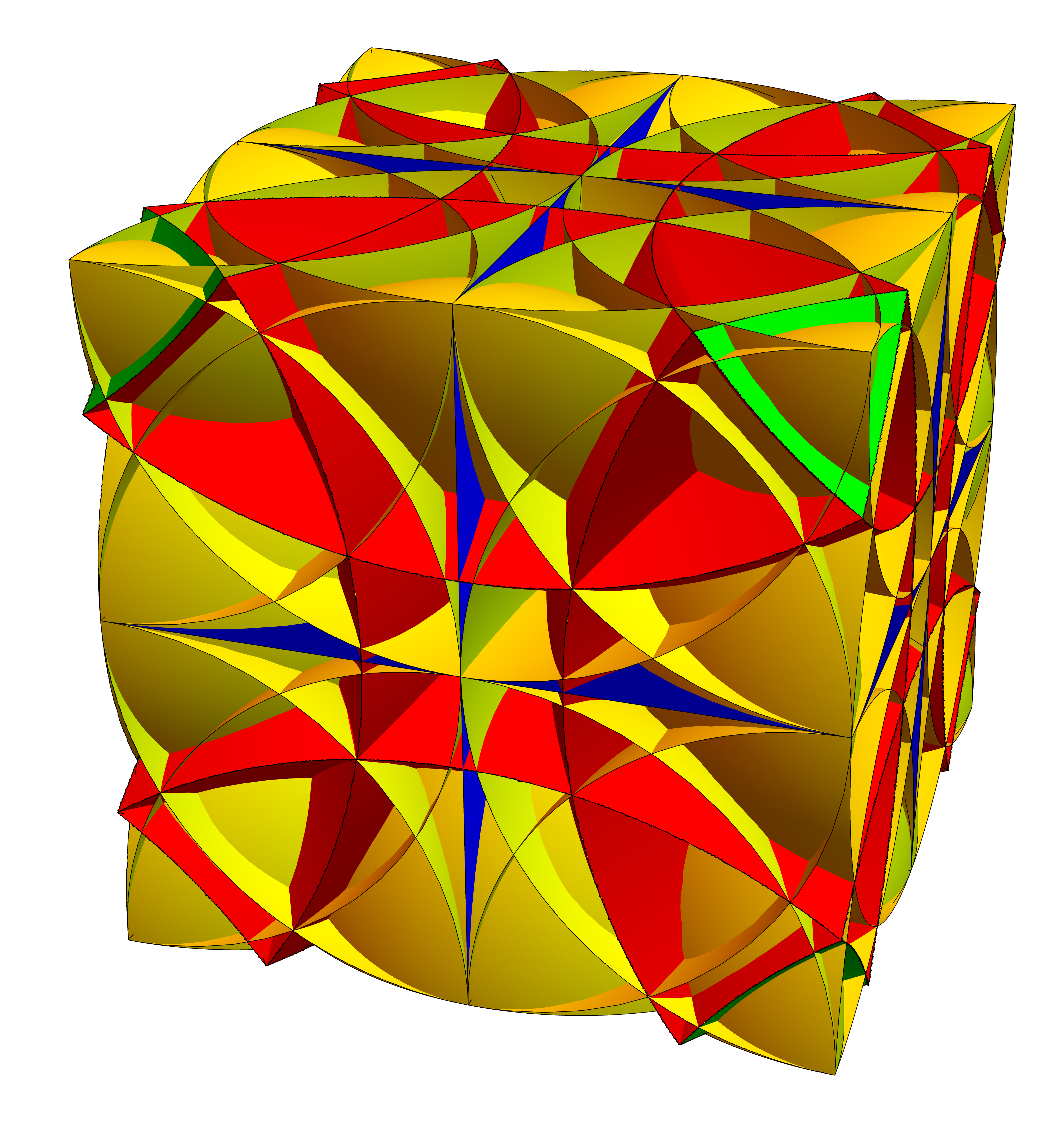}
         \caption{The cube.}
     \end{subfigure}
     \hfill{}
     \begin{subfigure}[b]{0.45\textwidth}
         \centering
        \includegraphics[width=\textwidth]{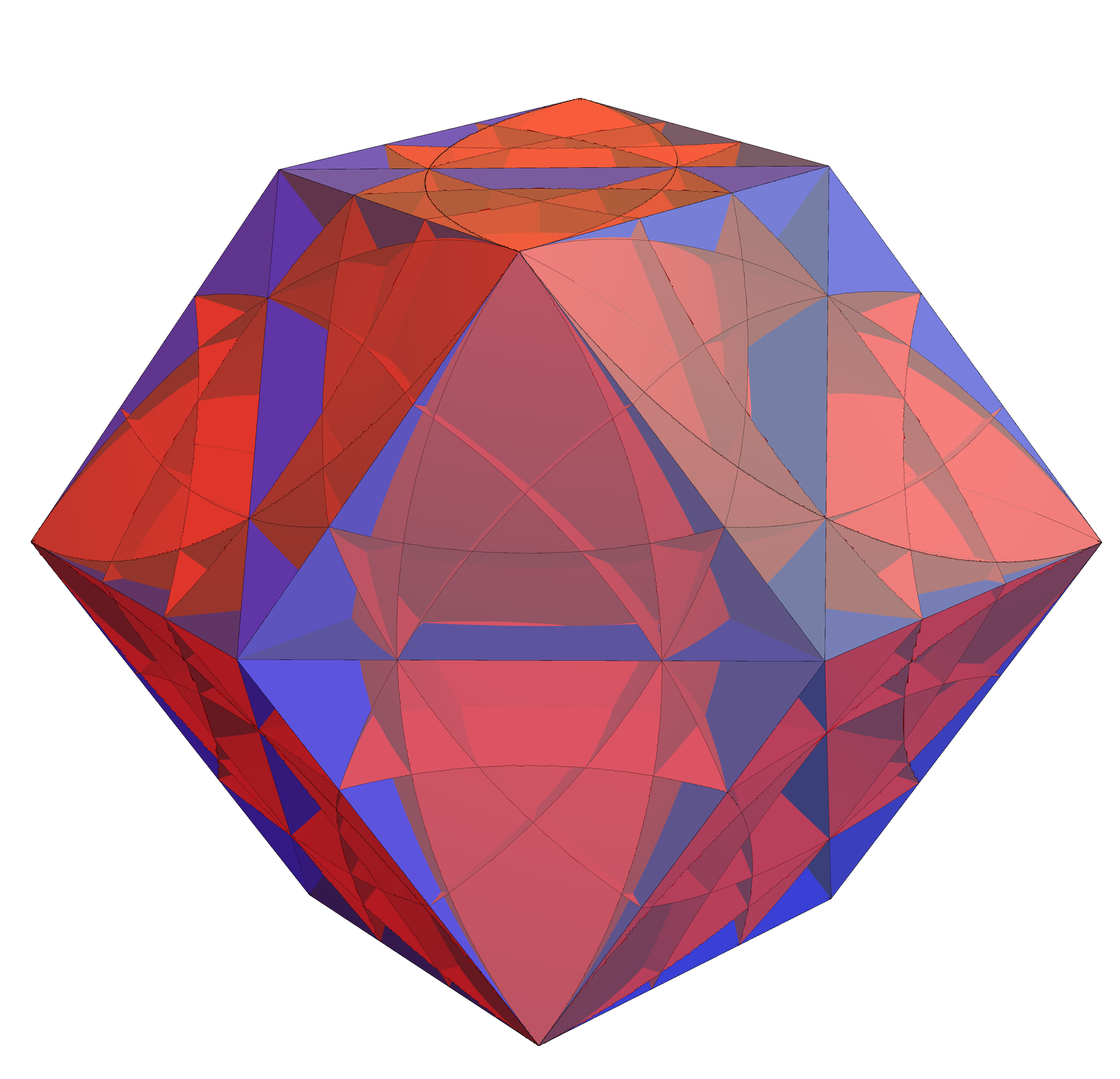}
         \caption{The rhombic dodecahedron.}
     \end{subfigure}
     \hfill{}
    \caption{Serendipitous decompositions of the cube and rhombic dodecahedron. The type 1-5 surfaces are, respectively: omitted (due to being the walls of the solid), blue, red, green, and yellow.}
    \label{fig:cube and dodecahedron}
\end{figure}

We will prove Theorem \ref{thm:main} using a combination of two methods, both relying on understanding the cylinder sets and the Jacobian $\omega_a$ of the inverse branches $T_a^{-1}$ of $T$. The first approach relies on black-box theorems of Nakada \cite{Nakada1976}, Nakada and Natsui \cite{NN}, and Schweiger \cite{Schweiger2000}, which ultimately rely on a direct measure-theoretic analysis of cylinder interactions. This provides us with exactness, CF-mixing, a Kuzmin-type theorem, and a piecewise-\emph{Lipschitz} invariant measure. To obtain piecewise-analyticity for the invariant measure, we turn to an argument of Hensley \cite{HensleyBook}, based on the work of Bandtlow--Jenkinson \cite{bandtlow2007invariant}, which ultimately relies on the theory of compact positive operators \cite{MR0181881}. The key idea here is to obtain compactness of the transfer operator on a restricted Banach space of functions that extend, in the appropriate sense, to holomorphic mappings on complexified neighborhoods of each piece of the serendipitous decomposition. This compactness is derived, ultimately, from the fact that $T$ is a uniformly expanding mapping.

\subsection{General Assumptions and Results}\label{sec:results}

Our general assumptions, making use of the framework of Iwasawa continued fractions \cite{lukyanenko_vandehey_2022}, are as follows. For our ambient space $X$, we will use $\R^d$, for $d\ge 1$; this includes the cases of the complex numbers $\mathbb{C}$ viewed as $\R^2$, the quaternions $\mathbb{H}$ viewed as $\R^4$, and the octonions $\mathbb{O}$ viewed as $\R^8$. For our inversion $\iota$, we will use a function $\iota:X\setminus\{0\}\to X\setminus\{0\}$ that satisfies  
\[
|\iota x| = \frac{1}{|x|} \qquad d(\iota x, \iota y) = \frac{d(x,y)}{|x||y|}, \label{eq:inversion formula}
\]
where $|\cdot|$ is the usual Euclidean norm and $d(\cdot,\cdot)$ is the usual Euclidean distance. Such inversions include the maps $z\mapsto 1/z$ in each of the division algebras, and the maps $x\mapsto x/\norm{x}^2$ in $\R^d$ for any $d$. Inversions are precisely the mappings of the form $\iota(x)=\mathcal O(x)/\norm{x}^2$ for some orthogonal mapping $\mathcal O$  (see Lemma 2.14 in \cite{lukyanenko_vandehey_2022}). We will assume for notational convenience that all inversions used in this paper are order-2, but this is not necessary for the results. For our lattice $\mathcal{Z}$, we take any discrete additive subgroup of $X$ with compact quotient. Finally for our continued fraction region $K$, we take the Dirichlet region
\(
K=\{x\in X \mst d(x,0)\leq d(x,z) \text{ for all } z\in \Zee\}
\)
of our lattice $\mathcal{Z}$, together with some choice of boundary so that for any $x\in X$, there is a unique element of $\mathcal{Z}$, denoted $[x]$, with $x-[x]\in K$. The above will be referred to as the general assumptions of the paper in any results that follow.

\begin{remark}
More broadly, the Iwasawa CF framework allows us to consider a non-Euclidean space $X$, a lattice $\Zee\subset \Isom(X)$, and any fundamental domain for $\Zee$. However, we will not work in this full generality because finite-range systems have not been identified in non-Euclidean settings.
\end{remark}

From the data $(X,\iota, \mathcal{Z},K)$, we can generate the continued fraction map $T:K\to K$ by 
\(
Tx = \begin{cases}
\iota x - [\iota x], & x\neq 0,\\
0, & x=0.
\end{cases}
\)
We can generate the continued fraction digits of $x$ by $a_i=a_i(x) = [\iota T^{i-1}x]$ provided $T^{i-1}x \neq 0$. Under very general conditions (e.g., the norm-Euclidean assumption listed below), the numbers
\(
\iota(a_1+\iota(a_2+\iota(a_3+\dots \iota(a_n) \dots )))
\)
will converge to $x$ (see \cite{lukyanenko_vandehey_2022}). The cylinder set $C_a$ for $a\in \mathcal{Z}$ is given by $C_a = K\cap \iota(K+a)$. If we let $T_a(x)= \iota(x)-a$ denote the injective restriction of $T$ to $C_a$, then we can  define $C_{a_1a_2 \dots a_n}$ by 
\(
C_{a_1a_2 \dots a_n} = T^{-1}_{a_1}T^{-1}_{a_2}\dots T^{-1}_{a_n} K.
\)
 Note that  $C_{a_1a_2 \dots a_n} = \emptyset$ is possible.

In this paper, we will be interested in lattices $\mathcal{Z}$ that satisfy the following properties:
\begin{defi}We will say a lattice $\Zee$ is:
\begin{itemize}
    \item \emph{integral} if $z\cdot z\in\mathbb{Z}$ for all $z\in\mathcal{Z}$. We will always treat $\mathbb{C}$ as $\mathbb{R}^2$, $\mathbb{H}$ as $\mathbb{R}^4$, and $\mathbb{O}$ as $\mathbb{R}^8$ for the purposes of calculating any dot product in this paper.
    \item \emph{unit-generated} if every $z\in \mathcal{Z}$ can be written as a sum of lattice points with norm $1$ (the units of the lattice).
    \item \emph{nicely invertible} with respect to $\iota$ if $|z|^2\iota(z)\in\mathcal{Z}$ for all non-zero $z\in\mathcal{Z}$, or equivalently $\mathcal O(\Zee)=\Zee$.
    \item \emph{norm-Euclidean} (and $K$ is \emph{proper}) if $\rad(K)<1$, where $\rad(K)=\sup\{\norm{x}:x\in K\}$.
    \item \emph{3-remote} if for any $z\in \mathcal{Z}$ of norm $\sqrt{3}$, we have that $d(z,K)\ge 1$,
    where $d(z,K):=\inf_{x\in K}d(z,x)$.
    \item \emph{7-remote} if for any $z\in\mathcal{Z}$ of norm $\sqrt{7}$, we have $d(z/2,K)\ge 1/2$
\end{itemize}
\end{defi}
Any lattice $\mathcal{Z}$ will be nicely invertible with respect to the inversion $\iota(x)=x/|x|^2$, so this condition is not onerous. The conditions of 3-remoteness and 7-remoteness are new to this paper, and arise from the analysis of cylinder sets and images of $\partial K$ under $T$.

Our main result will be the following.

\begin{thm}\label{thm:main}
Suppose that $X$ is $\R^d$, $\C$, $\mathbb{H}$, or $\mathbb{O}$, with a lattice $\Zee$ that is integral, unit-generated, norm-Euclidean, nicely invertible with respect to an inversion $\iota$, and either 3-remote or 7-remote. Let $K$ be the Dirichlet region for $\Zee$. Then:
\begin{enumerate}
    \item The Iwasawa CF given by $(X,\Zee,\iota,K)$ has the finite range property; that is, there is a finite collection of sets $U_1, \ldots, U_J\subset K$ such that any non-empty cylinder $C_{a_1a_2 \dots a_n}$ satisfies $T^{n}C_{a_1a_2 \dots a_n}=U_j \pmod 0$  for some $1\leq j\leq J$,
    \item The CF map $T$ is an exact endomorphism and continued fraction mixing with respect to a $T$-invariant measure $\mu$ that is equivalent to Lebesgue measure $\lambda$,
    \item The density $d\mu/d\lambda$ is piecewise analytic on the pieces of the partition generated by the sets $U_1,\ldots$, $U_J$, and
    \item $T$ satisfies a Kuzmin-type theorem (see \eqref{eq:GaussKuzmin}).
\end{enumerate}
\end{thm}

Most of this theorem will be an immediate consequence of four black box theorems, Theorems \ref{thm:black box 1}--\ref{thm:black box 4}. There are eight conditions that must be satisfied for these theorems to apply, which will be dealt with in Lemmas \ref{lemma:condition A},  \ref{lemma:condition C},  \ref{lemma:condition D},  \ref{lemma:Proving finite range},  \ref{lemma:condition H},  \ref{prop:All Uj contain full cylinders},  \ref{lemma:condition F}, and \ref{lemma:condition G}. The piecewise analyticity of the invariant measure will require a separate argument, provided in Theorem \ref{thm:real-analytic}.

\subsection{Examples}\label{sec:examples}

There are several classical cases, as well as several new cases, which are covered by Theorem \ref{thm:main}. In particular, all of the following satisfy its conditions (with the exception of (9), which requires a slightly more careful argument provided in Section \ref{sec:Z3}):
\begin{enumerate}
    \item The lattice $\mathcal{Z}=\mathbb{Z}$ on $\mathbb{R}$, with either $\iota(x)=1/x$ or $\iota(x)=-1/x$. Both systems are referred to as nearest integer continued fractions \cite{Rieger}; as is the CF with the inversion $x\mapsto \norm{1/x}$, which is not covered by our results.
    \item The lattice $\mathcal{Z}=\mathbb{Z}[\ii]$ on $\mathbb{C}$. Together with $\iota(z)=1/z$, this generates the A. Hurwitz continued fractions \cite{Hurwitz}. 
\end{enumerate}
    For all division algebras below, we will continue to use the mapping $\iota(z)=1/z$; other choices such as $\iota(z)=1/\overline z$ or $\iota(z)=-1/z$ may also be available. We will also use the usual basis for these spaces over $\R$, namely $\{1, \ii\}$ for $\C$, $\{1, \ii, \jj,\kk\}$ for $\mathbb{H}$, and $\{1, e_1, \ldots, e_7\}$ for $\mathbb{O}$, where $\ii^2=\jj^2=\kk^2=e_1^2=\ldots=e_7^2=-1$. For more information on the quaternions and octonions, see \cite{ConwayBook}.
\begin{enumerate}[resume]
    \item The lattice $\mathcal{Z}=\mathbb{Z}[\frac{1+\sqrt{3}\ii}{2}]$ on $\mathbb{C}$. Although the associated continued fraction algorithm has not been given a name, it also appears in the work of Hurwitz \cite{Hurwitz}.
    \item The lattice $\mathcal{Z}=\mathcal{H}$, the Hurwitz integers \cite{Mennen}, on $\mathbb{H}$, where 
    \(
    \mathcal{H}=\mathbb{Z}\oplus \mathbb{Z}\ii\oplus \mathbb{Z}\jj\oplus \mathbb{Z}\frac{1+\ii+\jj+\kk}{2}
    \)
    \item The lattice
    \(
    \mathcal{Z}=\mathbb{Z}\oplus \mathbb{Z}\ii\oplus \mathbb{Z}\frac{1+\sqrt{3}\jj}{2}\oplus \mathbb{Z}\frac{\ii+\sqrt{3}\kk}{2},
    \) the Gausenstein integers, on $\mathbb{H}$ (see \cite{Yoshii}). 
    \item The lattice 
    \(
        \mathcal{Z}=\mathcal{C} &= \mathbb{Z}\oplus \mathbb{Z}e_1\oplus \mathbb{Z}e_2\oplus \mathbb{Z}e_3\\
        &\qquad \oplus \mathbb{Z}h\oplus \mathbb{Z}e_1 h \oplus \mathbb{Z}e_2 h\oplus \mathbb{Z}e_3 h ,
    \)
    with $h=(e_1+e_2+e_3-e_4)/2$,  on $\mathbb{O}$. The integers $\mathcal{C}$ are known as the Cayley integers, and we provide the generators given by Rehm \cite{Rehm}.
\end{enumerate}For the next three examples, we will work with the somewhat unusual setting of $\R^3$ (see the Iwasawa continued fractions \cite{lukyanenko_vandehey_2022}), with the inversion given by $\iota(x,y,z)=\frac{1}{x^2+y^2+z^2}(x,y,z)$. Variations on it, such as $\iota(x,y,z)=\frac{1}{x^2+y^2+z^2}(-x,y,z)$, may also fit our conditions.   See \cite{CoxeterBook} for more information on these lattices. 
\begin{enumerate}[resume]        
    \item \label{example:quaternions} The lattice
    \(
    \mathcal{Z} = \mathbb{Z}\left( \frac{1}{\sqrt{2}},\frac{1}{\sqrt{2}},0\right) \oplus \mathbb{Z}\left( \frac{1}{\sqrt{2}},-\frac{1}{\sqrt{2}},0\right)  \oplus \mathbb{Z}\left( \frac{1}{\sqrt{2}},0,\frac{1}{\sqrt{2}}\right)  
    \)
    on $\mathbb{R}^3$. The Dirichlet regions for this lattice form the rhombic dodecahedral honeycomb, see Figure \ref{fig:cube and dodecahedron}.
    \item The lattice
    \(
    \mathcal{Z} = \mathbb{Z}\left( 1,0,0\right) \oplus \mathbb{Z} \left( \frac{1}{2}, \frac{\sqrt{3}}{2}, 0\right) \oplus \mathbb{Z}\left( 0,0,1\right)
    \)
    on $\mathbb{R}^3$. The Dirichlet regions for this lattice form the hexagonal prism honeycomb.
    \item The lattice $\mathbb{Z}^3$ on $\mathbb{R}^3$. The Dirichlet regions for this lattice form the usual cubic honeycomb.
\end{enumerate}
It is straightforward to check that all of the lattices above are integral and unit-generated (in most cases the generating elements are given explicitly). We will show all these lattices are norm-Euclidean in Lemma \ref{lemma:example lattices are norm-Euclidean} and that lattices (1)--(8) are 3-remote in Proposition \ref{prop: example lattices are 3-remote}. The lattice $\mathbb{Z}^3$ on $\mathbb{R}^3$ is not 3-remote, as $(1/2,1/2,1/2)$ is in $K$ but $(1,1,1)$ is a lattice point of norm $\sqrt{3}$, but it is 7-remote (see Section \ref{sec:Z3}).

We comment briefly on why some lattices are not in the list above. In $\mathbb{C}$, we could have considered  the lattices
\(
\mathbb{Z}[\sqrt{2}\ii],\hspace{.1in} \mathbb{Z}\left[ \frac{1+\sqrt{7}\ii}{2}\right] ,\hspace{.1in} \mathbb{Z}\left[ \frac{1+\sqrt{11}\ii}{2}\right],\hspace{.02in}\text{ or }\hspace{.02in} \mathbb{Z}[\sqrt{3}\ii].
\)
The first three are norm-Euclidean, but the fourth is not; however, none of these are unit-generated. In $\mathbb{H}$, we note that the Hurwitz integers and Gausenstein integers are two of the three possible norm-Euclidean orders \cite{Chaubert}. The third norm-Euclidean order is 
\(
\mathbb{Z}\oplus \mathbb{Z}\frac{2+\sqrt{2}\ii-\sqrt{10}\kk}{4}\oplus \mathbb{Z}\frac{2+3\sqrt{2}\ii+\sqrt{10}\kk}{4}\oplus \mathbb{Z}\frac{1+\sqrt{2}\ii+\sqrt{5}\jj}{2},
\)
see \cite{Fitzgerald-other}. However, this one is again not unit-generated: one can confirm this by calculating all the points of norm 1 and seeing that all of them have zero $\jj$ component.

In $\mathbb{R}^3$, the lattice
\(
\mathbb{Z}(1,0,0)\oplus \mathbb{Z}(0,1,0)\oplus \mathbb{Z}\left(\frac{1}{2},\frac{1}{2},\frac{1}{2}\right),
\)
whose Dirichlet regions are truncated octahedrons, is not unit-generated.

That said, we have not ruled out the possibility that there are other lattices that fit our conditions.

\subsection{Open Problems}

In this paper we have restricted ourselves to look at the spaces $\mathbb{R}^n, \mathbb{C}, \mathbb{H}, \mathbb{O}$, because in these spaces, translation by $\mathcal{Z}$ and inversion by $\iota$ will map spheres and hyperplanes onto spheres and hyperplanes. However, in other Iwasawa inversion spaces, these actions may distort these shapes and so the method of our proofs may not apply. It is unknown if there is any lattice on the Heisenberg group (see \cite{LV2015}) for which the associated continued fraction algorithm has the finite range property.

While norm-Euclidean quaternionic lattices have been well studied, little appears to be known about octonionic lattices. There may be other octonionic continued fraction algorithms that fit into the framework of this paper. 

Two other methods are commonly used to prove ergodicity for regular CFs, and it remains an open problem to extend them to our setting. The natural extension of $T$ (used in \cite{kraaikamp1991new}) has a known invariant measure, but in higher dimensions it now has a complicated fractal domain of definition \cite{EINN,HV}, making it difficult to pass results to $T$. Alternately, the argument via hyperbolic geometry and geodesic coding on a modular manifold comes closer to succeeding, but is hindered by the existence of hidden symmetries (namely, the stabilizer of $\infty$ in the modular group $\langle \Zee, \iota\rangle$ can be larger than the starting digit group $\Zee$). This phenomenon even appears in the A. Hurwitz CFs. Ergodicity follows for folded variants of these systems; for centrally-symmetric systems, as the ones under consideration, one obtains a \emph{finite number} of ergodic components  \cite{lukyanenko_vandehey_2022}.

\subsection{Structure of the Paper and Notation}  In Section \ref{sec:black box}, we will introduce our black box theorems. In Section \ref{sec:examples proofs}, we will prove that our lattices (except lattice (9)) satisfy the remaining conditions needed for Theorem \ref{thm:main} to apply. In Section \ref{sec:finite range}, we prove that finite range and serendipity are equivalent under broad conditions, and use this to show several perturbed examples are not serendipitous. In Section \ref{sec:proofs}, we will prove most of Theorem \ref{thm:main} by showing that the conditions of it suffice to show the conditions of the various black box theorems.  In Theorem \ref{sec:Z3}, we return to the case of lattice (9) and show that it satisfies the conclusion of Theorem \ref{thm:main} with weaker conditions. Finally in Section \ref{sec:real analytic}, we prove the final part of Theorem \ref{thm:main}, on the piecewise analyticity of the invariant measure.

We will frequently make use of asymptotic notations in this paper. By $f(x)=O(g(x))$  we mean that there exists some $C>0$ with  $|f(x)|\le C |g(x)|$ for all relevant $x$. We also write $f(x)\asymp g(x)$ if both $f(x)= O(g(x))$ and $g(x) = O(f(x))$ simultaneously. In more general contexts, we interpret $O(g(x))$ as referring to the set of functions $f(x)$ with $f(x) = O(g(x))$ and any equal signs are interpreted from left to right as either $\in$ or $\subset$ as appropriate. Thus, $\sin x = x+O(x^3)$ means that $\sin x$ belongs to the set of functions $x+O(x^3)$ (for $x$ small), and $(1+O(x))^{-1} = 1+O(x)$ means that all functions belonging to $(1+O(x))^{-1}$ also belong to $1+O(x)$ (for $x$ small).
\subsection{Acknowledgements}We would like to thank Martha Hartt for her comments on the paper.

\section{Preliminary Facts}
We start by recalling the black-box theorems about fibred systems in Section \ref{sec:black box}, and by connecting our examples to our Main Theorem \ref{thm:main} in Section \ref{sec:examples proofs}.

\subsection{Black Box Theorems}\label{sec:black box}

We now adapt general theorems about fibred systems to our setting, using the continued fraction notation described in Section \ref{sec:intro}.

Let $T_{a_1a_2\dots a_n} = T_{a_n}T_{a_{n-1}}\dots T_{a_2}T_{a_1}$ be the restriction of $T^n$ to the cylinder set $C_{a_1 \dots a_n}$ and let $\omega(a_1,\dots, a_n)$ be the Jacobian derivative of $T^{-1}_{a_1a_2\dots a_n}$, given by
\(
\omega(a_1,\dots,a_n)(x) = |\det DT^{-1}_{a_1a_2\dots a_n} x|, 
\)
satisfying, for any Borel set $E\subset K$,
\(
\int_E \omega(a_1,\dots, a_n) \ d\lambda = \int_{T^{-n}E \cap C_{a_1a_2\dots a_n}} d\lambda.
\)
For readability, we will often let $s$ denote the string $a_1a_2\dots a_n$, and then write $T_s$ for $T_{a_1\dots a_n} $, and $\omega_s$ for $\omega(a_1,\dots , a_n)$. 

We will work with the following conditions:

\begin{enumerate}
    \item[(A)] For each digit $a$ with $C_a$ non-empty, $T_a:C_a\to K$ is a one-to-one, continuous map with continuous first order partial derivatives and $\det DT_a\neq 0$.
    \item[(B)] The finite range property is satisfied. That is, there exist a finite number of positive-measure subsets $U_1, U_2,\dots, U_J$ of $K$ such that for each nonempty $C_{s}$, we have that $T^{|s|} C_{s}= U_j$ for some $j$. This equality may hold up to measure zero. \\
    We shall denote by $\mathcal{F}$ the partition of $K$ generated by the $U_j$'s, and refer to elements of $\mathcal{F}$ as \emph{cells}.
    \item[(C)] R{\'e}nyi's condition is satisfied: there is a uniform constant $L>1$ such that for all strings $s$, if $T^{|s|} C_{s } = U_j$ for some $j$, then 
    \[\label{eq: Renyis condition}
    \sup_{x\in U_j} \omega_s(x) \le L \inf_{x\in U_j} \omega_s(x).
    \]
    \item[(D)] Cylinders uniformly  shrink to $0$ in diameter as the number of digits increases. That is, letting 
    \(
    \sigma(m):=\sup_{|s|=m} \operatorname{diam} C_{s},
    \) 
    we have $\lim_{m\to \infty} \sigma(m)= 0$.
    \item[(E)] Each $U_j$ contains a full cylinder.
    \item[(F)] There is a constant $R_1>0$ such that for every finite digit sequence $s$ with $n=|s|$ and all $x,y\in  C_{s}$ we have
    \(
    \left|\omega_s(T^n x)- \omega_s(T^n y)\right|\le R_1 \lambda(C_{s})d\left( T^n x,T^n y\right).
    \)
\item[(G)]There is a constant $R_2>0$ such that for every $s$ with $n=|s|$ and all $x,y\in  C_{s}$ we have
    \(
    d( x ,  y) \le R_2 d(T^n x,T^n y).
    \)
    \item[(H)] Let $\mathcal{L}_m=\{s:|s|=m \text{ and }C_{s} \text{ is not contained in a cell } F\in \mathcal{F}\}$ and  $\gamma(m)=\sum_{s\in\mathcal{L}_m} \lambda(C_{s})$. We have $\lim_{m\to \infty} \gamma(m) = 0$.
\end{enumerate}

With these, we have the following.

\begin{thm}[Theorem 2 and 3 in \cite{Nakada1976}] \label{thm:black box 1}
Under conditions (A)--(E), there exists a unique probability measure $\mu$ on $K$ that is $T$-invariant and equivalent to the Lebesgue measure. Furthermore, $\mu(A) \asymp \lambda(A)$. $T$ is also an exact endomorphism with respect to $\mu$, and thus $T$ is mixing of all orders and ergodic.
\end{thm}

Note that the above theorem as it appears in \cite{Nakada1976} uses a different version of condition (E), but our version suffices as seen in \cite{Schweiger2000}.

\begin{thm}[Theorem 2 in \cite{NN}]
Under conditions (A)--(H), the mapping $T$ is continued fraction mixing. That is, define 
\(
\psi(m) = \sup\frac{|\mu(C_s\cap T^{-|s|-m}E) - \mu(C_s)\mu(E)|}{\mu(C_s)\mu(E)}, \quad m\in\mathbb{N},
\)
where $\mu$ is the invariant measure from Theorem \ref{thm:black box 1} and the supremum is taken over all positive-measure cylinder sets $C_s$ and Borel sets $E$ in $K$. Then $\sup_{m\geq 1} \psi(m)<\infty$ and $\lim_{m\to \infty} \psi(m) =0$. 
\end{thm}

\begin{thm}[Theorem 1 in \cite{Schweiger2000}]
Under conditions (A)--(H), there is a version $h$ of the invariant density $d\mu/d\lambda$, which is Lipschitz continuous on any cell $F\in\mathcal{F}$.
\end{thm}

For the purposes of our paper, the previous result is superseded by Theorem \ref{thm:real-analytic}, but we include it here for completeness.

\begin{thm}[Theorem 2 in \cite{Schweiger2000}] \label{thm:black box 4}
Let $A:L^1(\lambda)\to L^1(\lambda)$ be the transfer operator given by
\(
\int_{T^{-1}E} f \ d\lambda = \int_E (Af) \ d\lambda
\)
for $f\in L^1(\lambda)$. Let $\mathcal{L}$ be the class of functions $f:(0,\infty)\to(0,\infty)$ that are bounded away from both $0$ and $\infty$ and which are Lipschitz continuous on each cell $F$ of $\mathcal{F}$. Then under conditions (A)--(H), there is a constant $\alpha$ with $0<\alpha<1$ such that for any $f\in \mathcal{L}$, we have that
\[\label{eq:GaussKuzmin}
A^n f = \left(\int_K f \ d\lambda\right)h+O\left( \alpha^{\sqrt{n}}+\sigma(\sqrt{n})+\gamma(\sqrt{n})\right),
\]
where $h$ is as in the previous theorem, $\sigma$ is from condition (D), and $\gamma$ is from condition (H).
\end{thm}

For a fuller history of these types of black box theorems in relation to fibred systems, see also  \cite{Berechet,Renyi, SW,waterman, Zweimuller}.

\subsection{Facts about our Example Lattices} \label{sec:examples proofs}

\begin{lemma}\label{lemma:example lattices are norm-Euclidean}
Lattices (1)--(9) are norm-Euclidean.
\end{lemma}

We will show that these lattices are norm-Euclidean by finding extremal points of $K$. This is equivalent to finding the deep holes of the lattice, although we will not make use of this in the proof.

\begin{proof}
Several of these are well-known or such simple applications of geometry that we will not go into detailed proofs. Namely, for lattice (1), $\rad(K) = 1/2$. For lattice (2), $\rad(K) = 1/\sqrt{2}$, as $K$ here is a square centered at the origin with in-radius $1/2$. For lattice (3), $\rad(K) = 1/\sqrt{3}$, as $K$ here is a hexagon centered at the origin with inradius $1/2$. For lattice (4), $\rad(K)=1/\sqrt{2}$, which is proven in Theorem 3.4.3 of \cite{Mennen}. For lattice (6), $\rad(K)=1/\sqrt{2}$, which is proven in Theorem 2.2 of \cite{Rehm}. For lattice (9), $\rad(K) = \sqrt{3}/2$, which is due to the corner of the cube being at $(1/2,1/2,1/2)$.

For lattice (5), recall that our lattice in this case is
    \(
    \mathcal{Z}=\mathbb{Z}\oplus \mathbb{Z}\ii\oplus \mathbb{Z}\frac{1+\sqrt{3}\jj}{2}\oplus \mathbb{Z}\frac{\ii+\sqrt{3}\kk}{2}.
    \)
If we focus only on the real and $\jj$ dimensions, then the lattice is simply $\mathbb{Z}\oplus \mathbb{Z}[\frac{1+\sqrt{3}\jj}{2}]$, which has an extremal point at $(3+\sqrt{3}\jj)/6$, as this is just the standard hexagonal lattice. Likewise, on the $\ii$ and $\kk$ dimensions, lattice is simply $\ii\left(\mathbb{Z}\oplus \mathbb{Z}[\frac{1+\sqrt{3}\jj}{2}] \right) $, and there is an extremal point at $(3\ii+\sqrt{3}\kk)/6$. These two sublattices are perpendicular to one another, so an extremal point for $K$ is located at 
\(
\frac{3+3\ii+\sqrt{3}\jj+\sqrt{3}\kk}{6},
\)
which has norm $\sqrt{2/3}$.

Lattice (7) is given by
    \(
    \mathcal{Z} = \mathbb{Z}\left( \frac{1}{\sqrt{2}},\frac{1}{\sqrt{2}},0\right) \oplus \mathbb{Z}\left( \frac{1}{\sqrt{2}},-\frac{1}{\sqrt{2}},0\right)  \oplus \mathbb{Z}\left( \frac{1}{\sqrt{2}},0,\frac{1}{\sqrt{2}}\right)  .
    \)
For this lattice, the Dirichlet regions are all rhombic dodecahedrons, see Figure \ref{fig:cube and dodecahedron}. Rhombic dodecahedrons can be constructed via two equal cubes in the following manner: take one cube and cut it into 6 square pyramids with bases on the face of the cube and apex at the center of the cube; then attach these pyramids, by their bases, to the faces of the second cube (see \cite[pg.~26]{CoxeterBook}). The corners of the rhombic dodecahedron will be the 8 corners of the second cube and the 6 apexes of the pyramids. In our case, we can quickly calculate that these corners will be at $(\pm 1/2\sqrt{2}, \pm 1/2\sqrt{2},0)$ and all its permutations as well as $(\pm 1/\sqrt{2},0,0)$ and all its permutations. The latter points are the extremal points. Thus, this $K$ has radius $1/\sqrt{2}$.

Consider lattice (8). In this case, $K$ is a hexagonal prism. It is simple to check that in the first two coordinates, $(1/2, \sqrt{3}/6)$ is an extremal point as the lattice here is just the hexagonal lattice, and $1/2$ is an extremal point in the third coordinate as the lattice here is just $\mathbb{Z}$. Since these two sublattices are perpendicular to one another, an extremal point for $K$ is just $(1/2, \sqrt{3}/6,1/2)$, which has norm $\sqrt{21}/6\approx 0.76$.
\end{proof}

\begin{prop}\label{prop: example lattices are 3-remote}
Lattices (1)--(8) are all 3-remote.
\end{prop}

\begin{proof}
Recall that an integral lattice is 3-remote if for any $z\in\mathcal{Z}$ of norm $\sqrt{3}$, we have that $d(z,K)\ge 1$. This is trivially satisfied if $\rad(K)<\sqrt{3}-1$. We note the following values.
\(
\sqrt{3}-1&= 0.73205\dots\\
1/\sqrt{3} &= 0.57735\dots\\
1/\sqrt{2} &= 0.70711\dots
\)
As noted in the proof of Lemma \ref{lemma:example lattices are norm-Euclidean}, $\rad(K)$ is $1/2$ for lattice (1), $1/\sqrt{3}$ for lattice (3), and $1/\sqrt{2}$ for lattices (2), (4),  (6), and (7). So these are 3-remote.

Let us consider lattice (5). Here, $\rad(K) = \sqrt{2/3}\approx 0.816$, so the previous argument does not apply. So we will check all points of norm $\sqrt{3}$ and show that they are at least distance $1$ from $K$. There are 12 points of norm $\sqrt{3}$, which can be verified by computer calculation. These correspond to the points
\(
\frac{3+\sqrt{3}\jj}{2}\left( \frac{1+\sqrt{3}\jj}{2}\right)^n \qquad  \frac{\ii(3+\sqrt{3}\jj)}{2}\left( \frac{1+\sqrt{3}\jj}{2}\right)^n
\)
for $n=0,1,2,3,4,5$. We will show that the point $(3+\sqrt{3}\jj)/2$ is more than distance $1$ away from $K$, as the other cases are very similar. Note that since $1$ belongs to our lattice, $K$ must be completely contained in the half-space $x\cdot1 \le 1/2$. The distance from $(3+\sqrt{3}\jj)/2$ to this half-space is already $1$, so clearly the distance from this point to $K$ is at least $1$.

Finally, lattice (8) again has too large of a radius, as $\rad(K) = \sqrt{21}/6\approx0.76$. Here, we can identify this lattice with $\mathbb{Z}[\frac{1+\sqrt{3}\ii}{2}]\oplus \mathbb{Z}$, and  the points of norm $\sqrt{3}$ are precisely 
\(
\left( \frac{3+\sqrt{3}\ii}{2} \left( \frac{1+\sqrt{3}\ii}{2}\right)^n, 0\right)
\)
for $n=0,1,2,3,4,5$. As all cases are similar, we consider when $n=0$. Since our points of norm $\sqrt{3}$ all have last coordinate $0$ and $K$ is a prism, we need only work with the first two coordinates. However, this reduces down to the case of lattice (3), which has already been shown to be 3-remote.
\end{proof}

\section{Finite Range and Serendipity}\label{sec:finite range}
Fix an Iwasawa CF over $\R^d$. 

We now show that the finite range and serendipity conditions are equivalent if $K$ is bounded by finitely many hyperplanes and/or spheres (Lemma \ref{lemma:FRisSer}), and use this observation to study the finite range condition in $\alpha$-CFs.
    
Recall that for the empty string $\wedge$ we have $C_\wedge=K$ and, for a digit sequence $s$ and digit $a$, we have $C_{as}=T_a^{-1}C_{s}\cap K$. Set $E_0 = \partial K$, $E_{i+1} = E_i\cup TE_i$, and $E=\cupover_{i=0}^\infty E_i$.  We continue to work in Euclidean space (except for the slightly more general Lemmas \ref{lemma:boundaryIsInE} and \ref{lemma:Ebounds}) and use the convention $T(0)=0$ and $\iota(0)=0$, although in certain cases it is more convenient to think that $\iota(0)=\infty$.

We first show that boundaries of cylinders lie in $E$, and that, conversely, each point of $E$ lies in the boundary of a cylinder---in a quantitative way.
\begin{lemma}
\label{lemma:boundaryIsInE}
Consider any Iwasawa CF algorithm, $s$ a digit sequence, and $C_s$ the associated cylinder. Then $\partial T^{|s|}C_s\subset E_{|s|}\subset E$.
\begin{proof}
If $s$ is the empty string, then $T^0 C_s=K$, so $\partial T^0 C_s = E_0$.

Consider $TC_a$ for a digit $a$. By definition of cylinders, $C_a=\iota (K+a)\cap K$. Since $a$ is the first digit of all points in $C_a$, we have $T C_a=  \iota C_a -a= K \cap (\iota K -a)$. We then calculate
\(  
\partial T C_a &\subset  \partial K \cup  \left( \partial (\iota K-a) \cap K\right) \subset \partial K \cup  \left( \bigcup_{a_1\in\mathcal{Z}} \partial (\iota K -a_1)\cap K\right) \\
&= \partial K \cup  \left( \bigcup_{a_1\in\mathcal{Z}}  (\iota \partial K-a_1) \cap K\right) \subset \partial K \cup  T \partial K = E_1.
\)

In general, assume that $\partial T^{|s|}C_s\subset E_{|s|}$ for all $s$ of a fixed length $k$. We will show that this is also true for all $s$ of length $k+1$. Let $s$ be a string of length $k$ and $a$ a digit. Then by definition, $ C_{sa}$ consists of all points in $x\in C_s$ such that $T^k x \in C_a$, i.e., $T^k C_{sa} = C_a\cap T^k C_s$. But then, as we had in the previous paragraph
\(
T^{k+1}C_{sa} &= K\cap ( \iota K -a)\cap ( \iota (T^k C_s)-a)\\
\partial T^{k+1}C_{sa} 
   &\subset \partial K \cup  \left( \partial (\iota K -a)\cap K\right) \cup \left(  \partial (\iota (T^k C_s)-a) \cap K\right)\\
   &\subset E_1 \cup \left(  \partial (\iota (T^k C_s)-a) \cap K\right)
\subset  E_1 \cup \left( ( \iota \partial (T^k C_s)-a) \cap K\right) \\
&\subset  E_1 \cup \left( (\iota E_{|s|}-a) \cap K\right) 
   \subset E_1 \cup TE_{|s|} \subset E_{|s|+1}.
\)
This completes the proof by induction.
\end{proof}
\end{lemma}

\begin{lemma}
\label{lemma:Ebounds}
Every point of $E$ is in $T^{\norm{s}}C_s$ for some $s$: if $x\in \partial K=E_0$, then $s$ is the empty string of digits, and if $x\in E_n \setminus E_{n-1}$, then one can choose $s$ such that $n=\norm{s}$.
\begin{proof}
The case $x\in \partial K$ is immediate, so suppose $x\in E_{n}\setminus E_{n-1}$ and let $y\in \partial K$ with $T^{n} y = x$. If $T^iy=0$ for any $0\leq i\leq n-1$, then we would also have $x=0$, contradicting $x\in E_{n}\setminus E_{n-1}$. Thus, there is a string of digits $s=a_1a_2\dots a_{n}$ such that $T_{a_i}\dots T_{a_2}T_{a_1} y\in K$ for all $i$ with $1\le i < n$. Note that in fact $T_{a_i}\dots T_{a_2}T_{a_1} y\in K^\circ$ for $i< n$: otherwise, we would have  $z=T_{a_i}\dots T_{a_2}T_{a_1} y\in \partial K$ and $T^{n-i}z = x$ so $x\in E_{n-i} \subset E_{n-1}$. 
We claim that $y\in \partial C_s$. If we have that $y\in K$, then we have $y\in C_s$ and furthermore $y\notin C_s^\circ$ since $y\in \partial K$. If $y\notin K$, let $B'_\epsilon = B(y,\epsilon)\cap K$ for a small $\epsilon$ to be determined. We claim that for sufficiently small $\epsilon$ we have $B'_\epsilon \subset C_s$. Indeed, $B'_\epsilon \subset C_{a_1}$ if $\iota B'_\epsilon-a_1\subset K$. Since $\iota y-a_1\in K^\circ$, for sufficiently small values of $\epsilon$ it follows that $\iota B'_\epsilon-a_1\subset K^\circ$. Likewise reducing $\epsilon$, if necessary, to accommodate the remaining digits in $s$, we obtain a sufficiently small $\epsilon$ such that  $B'_\epsilon \subset C_s$, as desired. We may furthermore reduce $\epsilon$, if needed, to avoid $0\in T^i B'_\epsilon$ for $1\leq i\leq n-1$. From this, we get that $y\in \partial C_s$. We then observe that $y\in \partial C_s$ implies that $T^ny=x\in \partial T^nC_s$ since for both $y$ and $C_s$ we have $T^n=T_{a_n}\dots T_{a_2}T_{a_1}$, which is a homeomorphism on a neighborhood of $y$ and therefore preserves boundaries.
\end{proof}
\end{lemma}

We will next restrict our attention to objects constructed from hyperplanes and spheres (HASs)\footnote{While it is common to conflate both objects into the term ``sphere,'' we will want to reserve the term for metric spheres.}.  To prove the equivalence of our two conditions for regions bounded by finitely many HASs, we define a class of sets that includes such boundaries, and is also closed under finite unions, translations, and (when avoiding the origin) inversions.

\begin{defi}\label{defi:FSC}
A \emph{finite spherical complex (FSC)} is a set arising from the following construction.
Suppose $A_1, \ldots, A_n$ is a finite collection of codimension-1 HASs, $A=\cupover_i A_i \setminus \cupover_{i\neq j}A_i\cap A_j$ is their union with the pairwise intersections removed, and $\mathcal B$ is the set of closures of connected components of $A$. Let $\mathcal B' \subseteq \mathcal B$ and take $B=\cup \mathcal B'$. Then $B$ is called an FSC.
\end{defi}

\begin{example}
The boundary of the unit square is an FSC, as are all Dirichlet regions under consideration. FSCs are also closed under conformal mappings that don't send their points to $\infty$ and unions. In particular, we will prove that the sets in Figures \ref{fig:planar Hurwitz} and \ref{fig:cube and dodecahedron} are FSCs.
\end{example}

Finite spherical complexes possess the following key property:

\begin{lemma}
\label{lemma:FSCcomplement}
Let $A_1, \ldots, A_n$ be a collection of HASs. Then there are finitely many FSCs that can be constructed from $A_1, \ldots, A_n$, and the complement of any such FSC will have finitely many connected components.
\begin{proof}
Working in the one-point compactification of $\R^d$ (and compactifying each $A_i$ if it's a hyperplane), take any point $x$ of $U=A_1 \cup \cdots \cup A_n$ and send it to infinity using a M\"obius transformation. Then, any HASs that pass through $x$ are hyperplanes, and we obtain a neighborhood of $x=\infty$ that intersects $U^c$ with finitely many connected components. Now, by compactness of $U$, we obtain finitely many neighborhoods of $U$, each intersecting $U^c$ with finitely many components, whose union covers $U$. Together, this gives a neighborhood $U'$ of $U$ that intersects $U^c$ with finitely many connected components. Since any point of $U^c$ can be connected to some point of $U'$, we conclude that $U^c$ has finitely many connected components. 

Given an FSC constructed from $A_1, \ldots, A_n$, the complement of the FSC will contain the complement of $U$ as a dense subset, and will therefore have finitely many connected components. This gives the second claim of the lemma.

Now, the set $\mathcal B$ of building blocks for an FSC consists of the closures of connected components of $U\setminus \cupover_{i\neq j} A_i\cap A_j$. Equivalently, we may restrict our attention to each $A_k$ (thinking of it now as $\R^{d-1}$ or $\Sph^{d-1}$) and work with the components of $A_k\setminus \cupover_{i\neq k} A_i$, noting that each $A_k\cap A_i$ is either a sphere, a plane, point, or the empty set. The argument in the first paragraph then gives that each  $A_k\setminus \cupover_{i\neq k} A_i$ has finitely many connected components, so $\mathcal B$ is finite, giving the first claim of the lemma.
\end{proof}
\end{lemma}

\begin{lemma}
\label{lemma:FRisSer}
    Consider an Iwasawa CF algorithm with a fundamental domain $K$ that is bounded by finitely many HASs. Then the finite range condition is equivalent to serendipity. Furthermore, if either condition holds then $E$ is an FSC.
\end{lemma}
\begin{proof}[Proof of Lemma \ref{lemma:FRisSer}]
Suppose first that serendipity holds. Then, for any cylinder $C_s$ we have from Lemma \ref{lemma:boundaryIsInE} that $\partial T^{\norm{s}}C_s \subset E$. We then have that $K\setminus E$ contains no boundary points of $T^{\norm{s}}C_s$, so the interior of $T^{\norm{s}}C_s$ is a relatively-clopen set in $K\setminus E$, consisting of several of the components of $K\setminus E$. Since $K\setminus E$ has finitely many connected components by Lemma \ref{lemma:FSCcomplement}, there are finitely many options for what $T^{\norm{s}}C_s$ could be, up to a measure 0 sets along the boundary. Thus, the finite range condition holds. The fact that $E$ is an FSC is immediate from the fact that $\partial K$ is an FSC and the construction of $E$.

Conversely, let $F=\cupover_s \partial T^{\norm s}C_s$, where $s$ ranges over all digit sequences, and suppose now that the finite range property holds\footnote{A priori, the finite range condition only classifies the cylinders up to measure 0, but our cylinders are always bounded by finitely many HASs, so this measure 0 ambiguity disappears when we take their closures.}: only finitely many digit sequences contribute to union defining $F$, of length bounded above by some $N$. Combine Lemmas \ref{lemma:boundaryIsInE} and  \ref{lemma:Ebounds} to obtain $F=E_N$. Conclude that $E$ stabilizes, since for any $n\geq N$ we have $E_n=E_N$. It thus also follows that $E$ is an FSC, since $F$ is a union of cylinder-boundaries, which are FSCs, and the union of finitely many FSCs is an FSC. Lemma \ref{lemma:FSCcomplement} then provides that $K\setminus F$ has finitely many connected components, completing the proof of serendipity.
\end{proof}

We conclude that finite-range $\alpha$-CFs occur only when $\alpha$ is rational or a quadratic surd.

\begin{lemma}
\label{lemma:realserendipity}
Let $\alpha\in (0,1)$. Consider the $\alpha$-CF defined by the data $(\R, \Z, (-\alpha, 1-\alpha])$ with inversion $\norm{1/x}$ (Nakada $\alpha$-CFs \cite{nakada1981metrical}) or $1/x$ (Tanaka-Ito $\alpha$-CFs \cite{Tanaka-Ito1981}) or $-1/x$ (cf.~\cite{katok2010structure, lukyanenko_vandehey_2022}). If the system has the finite range property, then $\alpha$ is the root of a linear or quadratic polynomial over $\Z$.
\end{lemma}
\begin{remark}
One expects the converse to hold as well. For Nakada's $\alpha$-CFs, this follows from geodesic coding \cite{AS}.
\end{remark}
\begin{proof}
By Lemma \ref{lemma:FRisSer} (which also applies to Nakada's $\alpha$-CF, even though it is not an Iwasawa CF), the finite range condition implies that both points $-\alpha$ and $1-\alpha$ have finite orbits under the CF mapping $T$. Thus, these orbits are eventually periodic, and the tail of the orbit either arrives at 0 or is fixed by an element of $PSL(2,\Z)$. As in the classical case for regular CFs, this implies that both $-\alpha$ and $1-\alpha$ are roots of quadratic equation (possibly degenerate if the orbit reaches 0).

\end{proof}

Looking at subsystems, we obtain the following higher-dimensional corollary (for simplicity, we state the case of A.~Hurwitz CFs):

\begin{cor}
\label{cor:aphahurwitz}
Consider the $\alpha$-perturbed A.~Hurwitz CF, with data $(\R^2, \Z^2, (-0.5,.5]\times (-\alpha, 1-\alpha], \iota(x,y)=(x,-y)/(x^2+y^2))$. If $\alpha$ is not a root of a quadratic polynomial over $\Z$, then the system is not serendipitous (and the finite range condition fails).
\end{cor}

\begin{proof}
The system restricts to a copy of the real $\alpha$-CF along the imaginary axis, giving infinitely many distinct images of the point $(0,-\alpha)$. Furthermore, each of these is accompanied by an arc $A_i=T^i((-\epsilon_i, \epsilon_i)\times\{-\alpha\})$ with each $\epsilon_i>0$. Since $\iota$ is conformal (see Lemma 2.14 in \cite{lukyanenko_vandehey_2022}) and preserves the imaginary axis, each of the arcs $A_i$ is perpendicular to the imaginary axis. Since a circle intersects the imaginary axis at most twice, the set $\{A_i\}$ is in fact infinite and cannot be produced at a finite stage in the construction of the union $E=\cup T^i \partial K$.
\end{proof}

\section{Proof of the Main Theorem} \label{sec:proofs}

We will assume throughout this section that the general assumptions of the paper (see Section \ref{sec:results}) hold and specify when we require any additional assumptions of the Main Theorem \ref{thm:main}.

Let us start by sketching the ideas of the proof.

We start with the conditions (A), (C), and (D), which relate the properties of the mapping $T$, its iterated inverse branches $T_s^{-1}$, and their Jacobians $\omega_s$. We observe that $T_s$ is a composition of translations, which are isometric and have Jacobian 1, and inversions, which are conformal and are controlled by the inversion identity
$$d(\iota y, \iota z) = \frac{d(y,z)}{\norm{y}\norm{z}}.$$
Conformality means that, infinitesimally, distances are distorted equally in all directions, so that we may use the distance identity to calculate the Jacobian of $\iota$. This then allows us to explicitly write the $\omega_s (T^n x)$ in terms of the points $\{T^i x\}_{i=0}^{n-1}$. Conditions (A), (C), (D) follow from these considerations. Condition (G) is a straightforward consequence of the inversion identity.

We then study the structure of the cylinder sets and conditions (B), (E), and (H), by looking at the image of the boundary of the Dirichlet region $K$. We show that the finite range property (B) is satisfied by proving the equivalent condition of serendipity: we will show that the sequence of sets $E_i$ is eventually constant and is in fact a finite spherical complex (FSC). Under our assumptions, the boundary of $K$ is given by (subsets of) hyperplanes of the form $x\cdot z=1/2$ for $z\in \Zee$ of norm $1$. We refer to the set of these hyperplanes as type-1 objects. We show that iterates of $\partial K$ under $T$ can be decomposed into subsets of type-1 objects, as well as certain hyperplanes through the origin (type-2 objects) or certain unit spheres (type-3 objects). Since there is a finite number of type-1, type-2, and type-3 objects in total, there are finitely many FSCs that can be constructed from them, and so the sequence $\{\bigcup_{i=0}^j T^i(\partial K)\}_j$ eventually stabilizes. We conclude that for any cylinder $C_s$, we must have that $T^{\norm{s}}C_s$ is bounded specifically by these objects, giving a finite number of possibilities for $T^{\norm{s}}C_s$, and providing the finite range property (B). Properties (E) and (H) also follow from these considerations.

We finish by looking at condition (F), which is proven by combining previous arguments with a somewhat unexpected use of rational approximates. 

In each of the results below, we specify the assumptions used in the proof, which don't always correspond to the full assumptions of Theorem \ref{thm:main}. In particular, the proof of the finite-range property (B) applies to certain fundamental domains $K$ that are not Dirichlet domains for the given lattice, but nonetheless are bounded by type-1, type-2, and type-3 regions. We show some of these in Figure \ref{fig:notBatman}.

\begin{figure}
    \centering
    \includegraphics[width=.3\textwidth]{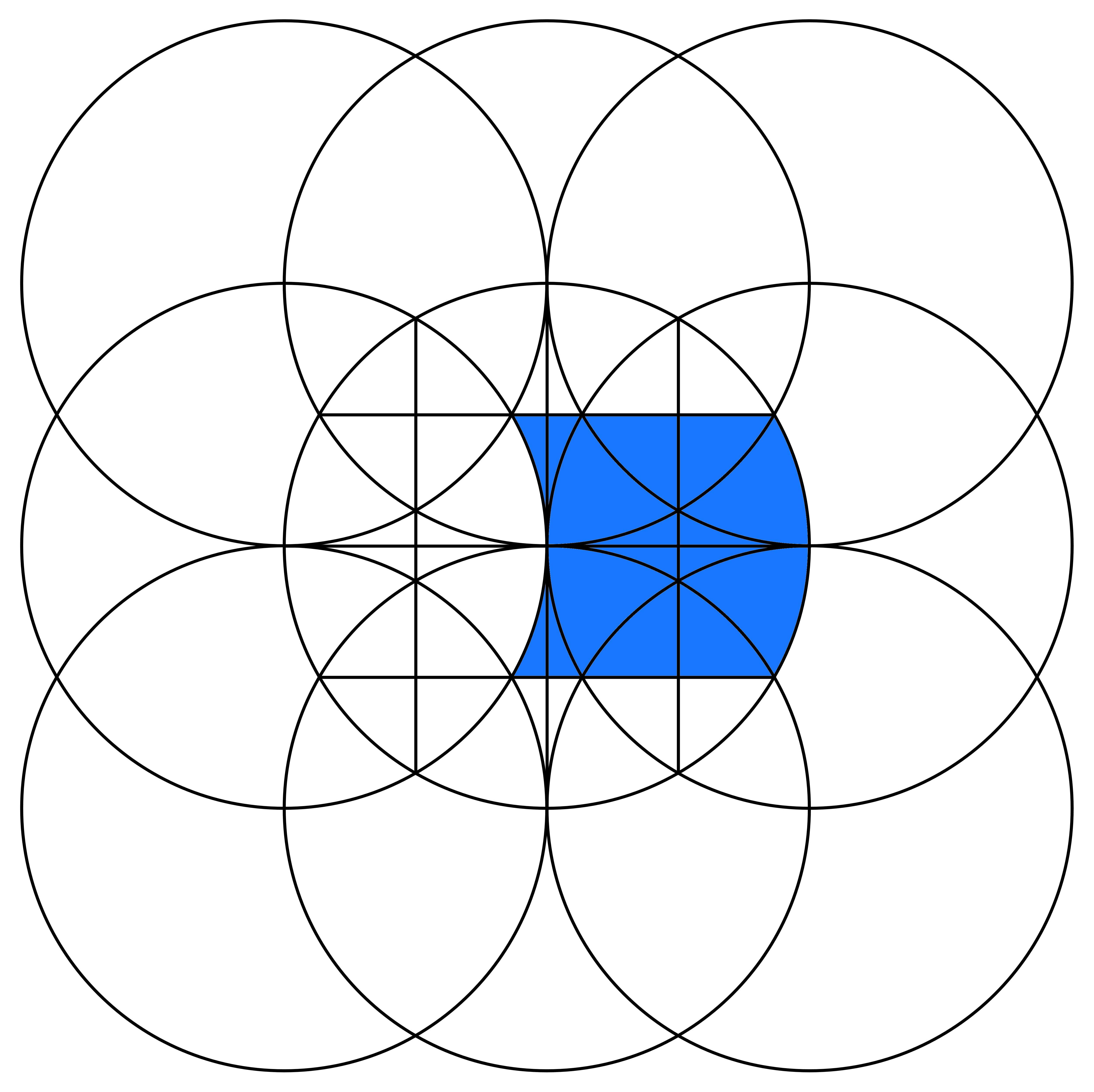}
    \includegraphics[width=.3\textwidth]{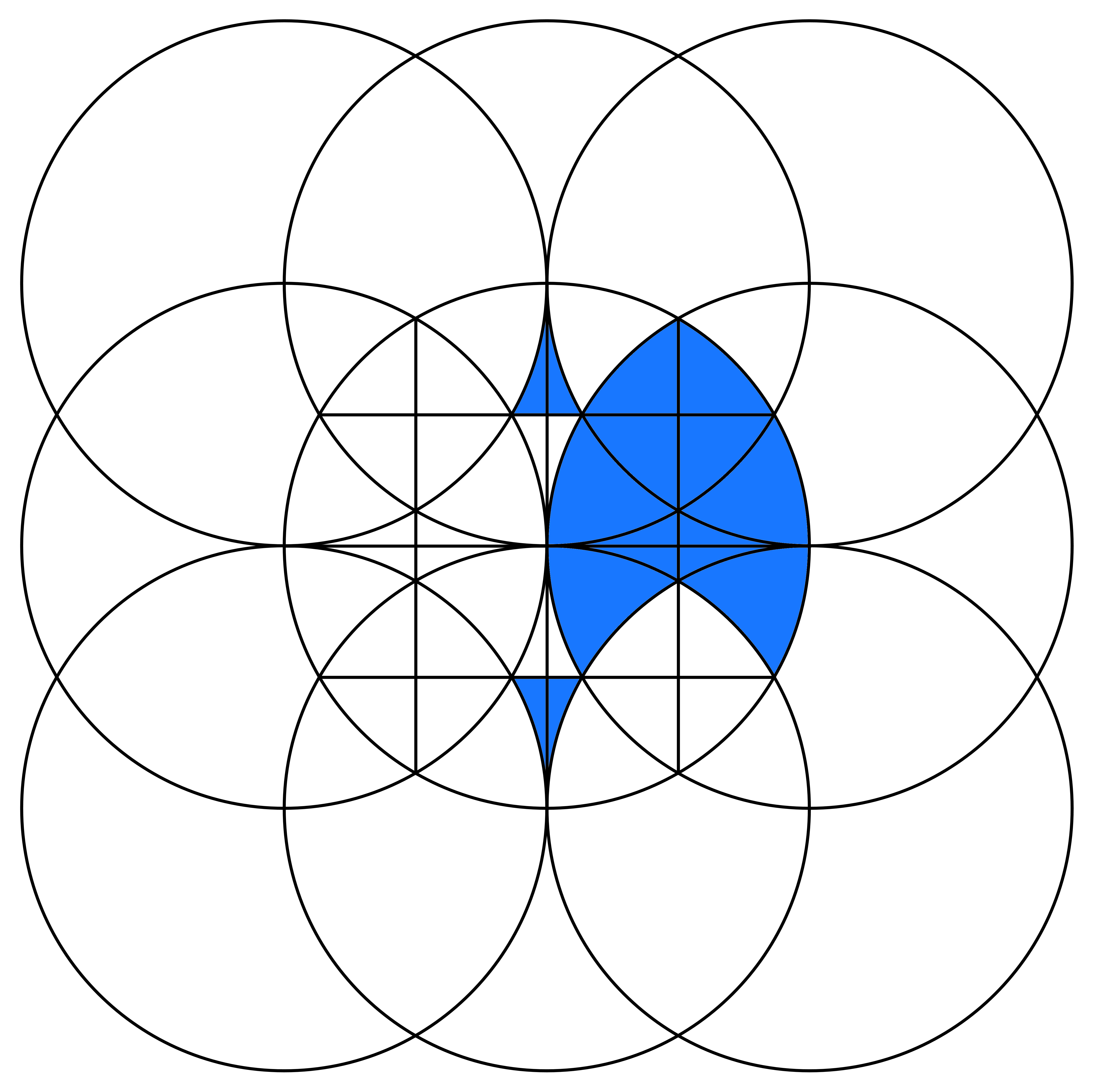}
    \includegraphics[width=.3\textwidth]{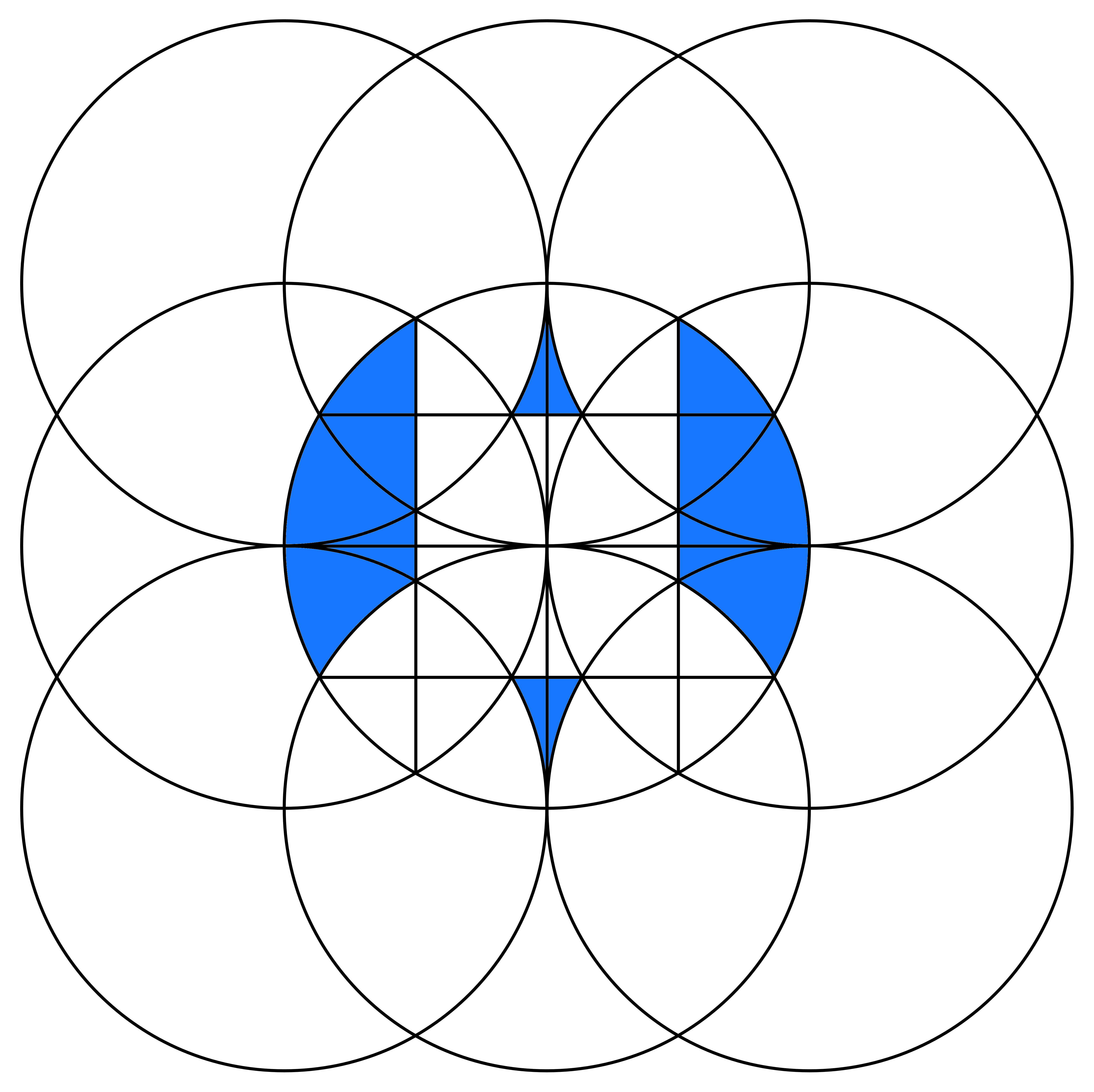}
    \caption{In addition to the Dirichlet region $[-.5,.5)\times [-.5,.5)$, one obtains the finite range property (B) for any fundamental domain bounded by the type-1, type-2, and type-3 regions. We show three such regions $K$ (blue subset) for $\Zee=\Z^2$, including the region corresponding to the chevron CFs (left figure) \cite{lukyanenko_vandehey_2022}. All circles shown have radius 1, and the figures are centered at the origin.}
    \label{fig:notBatman}
\end{figure}

\subsection{Properties of the Jacobian: Conditions (A), (C), (D), (G)}

\begin{lemma} \label{lemma:condition A}Condition (A) is satisfied.
\end{lemma}

\begin{proof}
The fact that the mappings $T_a:C_a\rightarrow K$ are one-to-one and continuous follows immediately from the definitions of $T_a$ and $C_a$. Any inversion $\iota$ can be seen as the composition of a orthogonal transformation with the inversion $x\mapsto x/|x|^2$, thus $T_a$ has continuous first order partial derivatives. The fact that $\det DT_a$ is non-zero is a special case of our next lemma.
\end{proof}

\begin{lemma}
\label{lemma:inversiondistortion} Let $s$ be a string with $|s|=n$. For $x\in  C_{s}$, we have that
\(
\omega_s(T^n  x) = \prod_{i=1}^{n} |T^{i-1} x|^{2d}
\)
where $d$ is the dimension of the ambient space over $\mathbb{R}$.
\end{lemma}

\begin{proof}
    We note that $\omega_s=\norm{\det DT^{-1}_{s}}$ measures the volume distortion of the mapping $T^{-1}_{s}$. This mapping is composed of translations, which do not alter volume, and inversions. By the inversion  identity \ref{eq:inversion formula}, for any point $y$ we have that $\iota(B(y,\epsilon))$ is approximated by $B(\iota(y), \norm{y}^{-2}\epsilon)$, for a distortion factor of $|y|^{-2d}$.  Thus, for a single digit $a$, we have $D_yT^{-1}_a=\norm{y+a}^{-2d}=\norm{\iota(y+a)}^{2d}=\norm{T_a^{-1}(y)}^{2d}$.    
    If $s=a_1a_2 \dots a_n$, then by the chain rule, we obtain:
\(
    |\det D_{T^n x}T^{-1}_{s} | &=|\det D_{T^n x}(T^{-1}_{a_1}T^{-1}_{a_{2}}\dots T^{-1}_{a_n}) |\\
    &= \prod_{i=1}^n |\det D_{T^{i} x}T^{-1}_{a_i} |\\
    &= \prod_{i=1}^n |T^{i-1}x |^{2d},
\)
    as desired.
\end{proof}

\begin{lemma} \label{lemma:condition C}
Assume $\Zee$ is norm-Euclidean. Then condition (C), R{\'e}nyi's condition, is satisfied.
\end{lemma}

\begin{proof}
Let $s$ be a string with $|s|=n$, and let $x, y\in C_{s}$. By reindexing Lemma \ref{lemma:inversiondistortion}, we have
\(
\frac{\omega_s(T^n x)}{\omega_s(T^n y)} = \left(\prod_{i=0}^{n-1} \frac{\norm{T^ix}}{\norm{T^iy}}\right)^{2d}.
\)
By the triangle inequality we have:
\(
\frac{\norm{T^i x}}{\norm{T^i y}} &\leq 1 + \frac{d(T^i x, T^i y)}{\norm{T^i y}}.
\)
Furthermore, for $i\neq n$, we apply the inversion formula \eqref{eq:inversion formula} repeatedly to obtain
\[
d(T^ix, T^i y) &= d(\iota T^i x, \iota T^i y) |T^i x||T^i y| \\&= d(\iota T^i x- a_{i+1},\iota T^i y - a_{i+1} )|T^i x||T^i y|\notag\\&=  d(T^{i+1}x,T^{i+1}y) |T^i x||T^i y| \notag  \\
&= d(T^{i+2}x,T^{i+2}y) |T^i x||T^i y|  |T^{i+1} x||T^{i+1} y| \notag \\
&=\cdots \notag \\
&= d(T^n x, T^n y ) \prod_{j=i}^{n-1} |T^j x||T^j y|.\label{eq:repeated inversion formula} 
\]
Thus we have that
\(\frac{d(T^i x, T^i y)}{\norm{T^i y}}&\leq d(T^n x, T^n y) \text{rad}(K)^{2(n-i)-1} \leq 2 \text{rad}(K)^{2(n-i)}.
\)
Returning to the Jacobian, we have
\(
\frac{\omega_s(T^n x)}{\omega_s(T^n y)} &=  \left(\prod_{i=0}^{n-1} \frac{\norm{T^ix}}{\norm{T^iy}}\right)^{2d}     \leq \prod_{i=0}^{n-1} \left(1+2 \text{rad}(K)^{2(n-i)}\right)^{2d}\\
    &\leq \exp \left(2d \sum_{i=0}^{n-1}2 \text{rad}(K)^{2(n-i)}\right)  = \exp \left( 4d \sum_{i=1}^{n} \text{rad}(K)^{2i}\right)\\
    &\leq \exp \left( 4d \sum_{i=1}^{\infty}\text{rad}(K)^{2i}\right) = \exp \left( \frac{4d \text{rad}(K)^2}{1-\text{rad}(K)^2}\right),
\)
which is finite and uniform over all $s$ as desired.
\end{proof}

\begin{lemma}\label{lemma:condition G}
Assume  $\mathcal{Z}$ is norm-Euclidean. Then condition (G) is satisfied.
\end{lemma}
\begin{proof}
Let $s$ be a string with $|s|=n$, and let  $x,y\in C_s$. Repeated applications of the inversion formula as in \eqref{eq:repeated inversion formula} give
\(
d(x,y) = \left( \prod_{i=0}^{n-1} |T^i x||T^i y| \right)d(T^{n}x,T^{n}y) < d(T^{n}x,T^{n}y),
\)
as desired.
\end{proof}

\begin{lemma} \label{lemma:condition D}
Assume $\Zee$ is norm-Euclidean. Then condition (D) is satisfied with
\(
\sigma(m) \le 2 \rad(K)^{2m+1}
\)
\end{lemma}

\begin{proof}
For $x,y\in C_{s}$ with $|s|=n$, \eqref{eq:repeated inversion formula} applied with $i=0$ gives
\(
d(x,y) \le 2 \rad(K)^{2n+1}.
\tag*{\qedhere} 
\)
\end{proof}

\subsection{Structure of Cylinder Sets: Conditions (B), (E), and (H)} 

We showed in Lemma \ref{lemma:FRisSer} that serendipity implies the finite range condition (B).

We next use the results of Conway-Sloane \cite{CS} to show that under our assumptions the Dirichlet region $K$ for $\Zee$ is the intersection of half-spaces  corresponding to the unit-norm generators of $\Zee$,
giving us a concrete description of $K$.

To begin with, we need a useful fact.

\begin{lemma}\label{lemma:half-directional}
Suppose $\mathcal{Z}$ is an integral lattice. Then for any $z,w\in \mathcal{Z}$, we have $2z\cdot w\in \mathbb{Z}$.
\end{lemma}

One interesting geometric consequence of this is that any two units are a multiple of 60 degrees or a multiple of 90 degrees from each other. In particular, the cross section of the lattice generated by two non-collinear units should be one of the two classical integral lattices on $\mathbb{R}^2$, either the square lattice, or the triangular lattice.

\begin{proof}
This follows from the fact that for any $z,w\in \mathcal{Z}$, we have that
\(
|z+w|^2 = |z|^2+2z\cdot w + |w|^2,
\)
and the fact that for any integral lattice, all norm-squares are integers.
\end{proof}

\begin{lemma}\label{lemma:ConwaySloane}
Let $\mathcal{Z}$ be an integral, unit-generated lattice. Then $\partial K$ consists of subsets of the hyperplanes $\{x: x \cdot z=1/2\}$ where $z\in \Zee$ ranges over the unit-norm elements of $\Zee$.
\end{lemma}

\begin{proof}
For each $r\in \mathbb{Z}_{\ge0}$, let $\Lambda_r=\{z\in \mathcal{Z}: z\cdot z = r\}$. 
By Theorem 9 of \cite{CS}, if we can show that $\Lambda_r \subset \Lambda_1+\Lambda_1+\dots+\Lambda_1$ (where there are $r$ terms in the sum), then the desired result holds. (The second condition of Theorem 9 of \cite{CS} is trivial for integral lattices.)

Fix $r>0$ and consider $z\in \Lambda_r$. Since $\mathcal{Z}$ is unit-generated and abelian, we know that we can express \(z=\sum_{i=1}^\ell c_i u_i, \qquad c_i\in \mathbb{Z}, \qquad u_i\in \Lambda_1\) as a linear combination of $\mathbb{Z}$-linearly independent units. We shall choose these vectors in a particular way. First of all, we may assume that all $c_i$ are non-negative: if any $c_i$ is negative, we can replace $c_i$ with $-c_i$ and $u_i$ with $-u_i$. 

We furthermore claim that we can choose the $u_i$'s so that $u_i\cdot u_j\ge 0$ whenever $c_ic_j>0$. Observe first that $u_i\cdot u_j \in\{-1,-1/2,0,1/2,1\}$ by Lemma \ref{lemma:half-directional}, and $\Z$-linear-independence further implies that for  $i\neq j$ we have $u_i\cdot u_j \in\{-1/2,0,1/2\}$. We therefore only need to resolve the case when $u_i\cdot u_j=-1/2$. In this case, we have that $u_i+u_j$ is again a unit, and replace either $u_i$ or $u_j$ with $u_i+u_j$ in the following way (recalling that $c_ic_j>0$):
\[
c_iu_i + c_ju_j = \begin{cases}
c_i(u_i+u_j)+(c_j-c_i) u_j, & \text{if }c_j\ge c_i,\\
c_j(u_i+u_j)+(c_i-c_j) u_i, & \text{if }c_i\ge c_j.
\end{cases}
\]
We note that this process will not alter the linear independence of the set of unit vectors. Moreover, this process shrinks the product $c_ic_j$ to $c_ic_j-c_i^2$ or $c_ic_j-c_j^2$ as appropriate.

We can repeatedly apply the replacement process in the previous paragraph so long as we find $u_i\cdot u_j=-1/2$ with $c_ic_j>0$. This process must eventually terminate due to the products $c_ic_j$ being non-negative integers that get smaller every time we iterate the replacement process.

So now let us assume that $z=\sum_{i=1}^\ell c_iu_i$ with $c_i\in \mathbb{Z}_{\ge 0}$ and $u_i\cdot u_j\ge 0$ whenever $c_ic_j>0$. Then we have
\(
r&= z\cdot z = \left( \sum_{i=1}^\ell c_iu_i\right)\cdot \left(\sum_{i=1}^\ell c_iu_i\right)\\
&= \sum_{i=1}^\ell c_i^2 + \sum_{i\neq j} 2c_ic_j u_i\cdot u_j\\
&\ge \sum_{i=1}^\ell c_i^2 \ge \sum_{i=1}^\ell c_i,
\)
where the last line holds because for non-negative integers $c_i^2\ge c_i$. Since each $u_i$ is a unit, $\sum_{i=1}^\ell c_i$ represents a number of unit vectors that can be added together to reach $z$. Thus, since $r\ge \sum_{i=1}^\ell c_i$ every $z\in \Lambda_r$ can be written as a sum of at most $r$ unit vectors, as desired.
\end{proof}

We can now prove property (B):
\begin{lemma}\label{lemma:Proving finite range}
Suppose a lattice $\mathcal{Z}$ is integral, nicely invertible with respect to an inversion $\iota$, norm-Euclidean, unit-generated, and 3-remote.
Then the finite range property is satisfied.
\end{lemma}

\begin{proof}
    Since $K$ is a Dirichlet region for a discrete group, it is bounded by hyperplanes; and Lemma \ref{lemma:ConwaySloane} tells us that these are of the form $x\cdot z=1/2$ for units $z\in \Zee$. By Lemmas \ref{lemma:FSCcomplement} and \ref{lemma:FRisSer}, we need only show that the sets $E_i$ will eventually stabilize to a finite spherical complex (FSC) (see Definition \ref{defi:FSC}).
    
    In fact, we will show that the $E_i$'s are contained (possibly strictly, as in the case of Hurwitz CFs, Figure \ref{fig:planar Hurwitz}) in the intersection of $\overline K$ with the union of the following objects:
    \begin{itemize}
        \item (type-1) the hyperplanes $x\cdot z=1/2$ with $z\in \mathcal{Z}$ of norm $1$,
        \item (type-2) the hyperplanes $x\cdot z=0$ with $z\in \mathcal{Z}$ of norm $1$,
        \item (type-3) the spheres $S_1(z)$ with $z\in\mathcal{Z}$ of norm $1$ or $\sqrt{2}$.
    \end{itemize}
    Furthermore, from this we will show that each $E_i$ will be an FSC formed from the above HASs.
    
    The set $E_0=\partial K$ is bounded by type-1 objects by Lemma \ref{lemma:ConwaySloane}. Thus it suffices to show that if we apply $T$ to a point in any of these objects, the result is contained in the union of all the objects. We will once again work on the one-point compactification $\R^d\cup \infty$, so that $\iota$ interchanges $0$ and $\infty$. We will also use the fact that $\iota$ is conformal, so that it sends HASs to HASs.
    
    To begin with, consider $z\in\mathcal{Z}$ of norm $1$ and the hyperplane $x\cdot z=1/2$. The point on this hyperplane nearest the origin is $z/2$. If we invert this hyperplane by applying $\iota$, we must then end up with a sphere through the origin whose point farthest from the origin is $2\iota(z)$. In other words, this sphere is $S_1(\iota(z))$, and by our assumption of nice invertibility $\iota(z)$ is in $\mathcal{Z}$. 
    Thus, given a point $x$ in the hyperplane $x\cdot z=1/2$, we have that $\iota (x) \in S_1(\iota(z))$ and so $T(x)\in S_1(z')$ for some $z'\in \Zee$. Now, 3-remoteness implies that if $\norm{z'}\ge 3$, then $\overline K \cap S_1(z')$ is at most a single point in $\partial K$. In this case, $T(x)$ is in a type-1 object. Otherwise $\norm{z'}=1$ or $\norm{z'}=\sqrt{2}$, and $T(x)$ is in a type-3 object.
    
    Next consider a sphere $S_1(z)$, where $z\in\mathcal{Z}$ has norm $\sqrt{2}$. (We will return to the case of norm $1$ later.) Let $z^* = z/|z|$. The point on $S_1(z)$ nearest the origin is at $(\sqrt{2}-1)z^*$ and the point farthest from the origin is at $(\sqrt{2}+1)z^*$. Note that $(\sqrt{2}-1)^{-1}=\sqrt{2}+1$. Thus, $\iota S_1(z)$ is a sphere whose point nearest to the origin is $(\sqrt{2}-1)\iota( z^*)$ and whose point farthest from the origin is $(\sqrt{2}+1)\iota (z^*)$. In other words, $\iota S_1(z)=S_1(|z|^2\iota(z))$, and furthermore the nicely-invertible assumption gives  $|z|^2\iota(z)\in \mathcal{Z}$. Thus, by the same argument made in the previous paragraph, points in $S_1(z)\cap \overline K$ are mapped by $T$ to points in type-1 or type-3 objects.
    
    Next consider a hyperplane $x\cdot z=0$ with $z\in\mathcal{Z}$ of norm $1$, which is the perpendicular bisector between $z$ and $-z$. In particular, the hyperplane is perpendicular to the line between $z$ and $-z$ both at $0$ and at $\infty$. Since $\iota$ preserves the unit sphere and distances for points on the unit sphere, it follows that $\iota(z)$ and $\iota(-z)$ are antipodes, and therefore $\iota(-z)=-\iota(z)$. Furthermore, the line through $z$ and $-z$ (as well as through $0$ and $\infty$) is sent to the line between $\iota(z)$ and $-\iota(z)$. Since $\iota$ is conformal, we conclude that the hyperplane is mapped to the hyperplane $x \cdot \iota(z)=0$.
    Now we want to consider what happens when we translate pieces of the inverted hyperplane $P$ (with normal vector $\iota(z)$) by elements of $\mathcal{Z}$ to return to $K$. For any $w\in \Zee$, only motion along the normal vector affects the position of the hyperplane, so we have that $d(P, w+P)=\iota(z)\cdot w$.  By Lemma \ref{lemma:half-directional}, this distance will be a multiple of $1/2$. Thus, translates of the hyperplane will have the form $x\cdot \iota(z)= a/2$, where $a\in \mathbb{Z}$. Negating $\iota(z)$ if necessary we may assume $a\geq 0$. This leaves the options of a type-2 hyperplane through the origin or a type-1 hyperplane along the boundary of $K$. Higher values of $a$ are ruled out since they correspond to points outside $\overline K$. 
    
    Finally, consider a sphere $S_1(z)$ with $z\in \mathcal{Z}$ of norm $1$. This is a sphere through the origin, so its inverse is a hyperplane. Moreover, since the farthest point on the sphere is $2z$, the nearest point on the hyperplane is the point $\iota(z)/2$, making this the hyperplane $x\cdot \iota(z)=1/2$. By following the method of the previous paragraph, we conclude that $T(S_1(z))$ must consist of planar objects $x\cdot \iota(z)=1/2$, $x\cdot \iota(z) = 0$ or $x\cdot (-\iota(z))=1/2$.

    We have thus shown that $E$ is a subset of the union of all type-1, type-2, and type-3 objects. It remains to show that $E$ is, in fact, a finite FSC. To this end, work inductively. By assumption, $E_0=\partial K$ is an FSC. For $i\geq 1$, we may write $E_i=E_{i-1}\cup \cupover_{z\in \Zee} \overline K \cap (z+\iota E_{i-1})$. Working with each $z\in \Zee$ individually, we observe that $\overline K \cap (z+\iota E_{i-1})$ is one the finitely-many FSCs generated by the type-1, type-2, and type-3 objects (Lemma \ref{lemma:FSCcomplement}). Thus, in constructing $E_{i}$ we are in fact taking the union of finitely many different FSCs that are generated by the same hyperplanes and spheres, and therefore obtain once again a FSC generated by the same HASs. In particular, by Lemma \ref{lemma:FSCcomplement}, there are finitely many options for what $E_i$ could be, and since $E_{i-1}\subset E_i$, the $E_i$'s must eventually stabilize, as desired.
\end{proof}

\begin{remark}
In Lemma \ref{lemma:Proving finite range}, the only way for type-2 objects to appear is if there exist $z_1,z_2\in \mathcal{Z}$ with $z_1\cdot z_2\not\in \mathbb{Z}$. This explains why type-2 objects did not appear in Figure \ref{fig:planar Hurwitz}.
\end{remark}

We can extend the above reasoning to provide condition (E): each $U_i$ contains a full cylinder.

\begin{lemma}\label{prop:All Uj contain full cylinders}
Suppose a lattice $\mathcal{Z}$ is integral, nicely invertible with respect to an inversion $\iota$, norm-Euclidean, unit-generated, and 3-remote. Then condition (E) is satisfied.
\end{lemma}

\begin{proof}
We make use of the key ideas of the proof of Lemma \ref{lemma:Proving finite range}. There, we started with the hyperplanes $x\cdot z=1/2$ for $z$ of norm $1$, because $\partial K$ belongs to the union of these hyperplanes. Now we will start with the half-spaces $x\cdot z<1/2$ for $z$ of norm $1$, because $K$ is the intersection of these half-spaces, up to boundary. Analyzing how $T$ maps these objects as we did before, we see that any set $U_j$ is (up to boundary) a non-empty intersection finitely many sets of the form
    \begin{itemize}
        \item half-spaces $x\cdot z< 1/2$ with $z\in \mathcal{Z}$ of norm $1$,
        \item half-spaces $x\cdot z< 0$ with $z\in \mathcal{Z}$ of norm $1$,
        \item sphere-exteriors $\{ x: d(x,z)> 1\}$ with $z\in\mathcal{Z}$ of norm $1$ or $\sqrt{2}$.
    \end{itemize}

    We want to show that a given $U_j$ contains a full cylinder. So consider $\iota U_j$. 
    
    Observe first that, by the description above, the closure of $U_j$ contains the origin, and for any $r>0$, $\lambda(B_r(0)\cap U_j)>0$. Thus, $\iota U_j$ has positive measure on neighborhoods of $\infty$, which we will now exploit.
    
    Observe next that $\iota U_j$ is bounded by hyperplanes or $S_1(z)$ with $\norm{z}=1$ or $\norm{z}=\sqrt 2$.  Near $\infty$ (in particular, outside of the ball $B_3(0)$) the spheres do not restrict membership in $\iota U_j$, and we may imagine that $\iota U_j$ consists of the intersection of half-spaces of the form $x\cdot z<1/2$ and $x\cdot z<0$ for various $z\in \Zee$. This intersection has positive measure by the above argument about neighborhoods of $\infty$. We may therefore take a (large) digit $a\in \Zee$ such that $(a+K)\cap (\iota U_j)$ has positive measure and furthermore $(a+K)\cap (\iota U_j)$ is in fact simply the intersection of $a+K$ with half-spaces of the form $x\cdot z<1/2$ and $x\cdot z<0$ for various $z\in \Zee$.

    If $a+K\subset \iota U_j$, then $T^{-1}_a K = \iota(a+K) \subset U_j$, and this would imply that $C_a$ is a full cylinder and contained in $U_j$. However, due to the possibility that $\iota U_j$ is caught between two close half hyperplanes $1/2 \ge x\cdot z\ge 0$, we cannot guarantee this. If this happens, let  $K'=K \cap (\iota U_j - a)\subset K$ so that $a+K' = a+K \cap \iota U_j$. Then $K'$ is an intersection of half-planes of the form $x\cdot z<1/2$ or $x\cdot z <0$ for $z\in\mathcal{Z}$ of norm $1$. In particular, $K'$ can be defined without using any sphere-exteriors.
    
    Inverting $K'$, we obtain a set $\iota K'$ that can be defined without any half-spaces of the form $x\cdot z<1/2$, i.e., it is the intersection of sphere-exteriors and finitely-many half-spaces $x\cdot z<0$ through the origin. Looking again outside of the set $B_3(0)$, we see a cone, which contains arbitrarily large open balls. In particular, there is a (large) digit $b$ such that $b+K\subset \iota K'$. Thus, 
    \(
    C_{ab}= T_{a}^{-1} C_b \subset T_a^{-1} K' \subset U_j,
    \)
    so $C_{ab}$ is a full cylinder inside $U_j$, as desired.
\end{proof}

With the above results about the cylinder sets, condition (H) is immediate from the structure of their boundaries.
\begin{lemma} \label{lemma:condition H}
Suppose a lattice $\mathcal{Z}$ is integral, nicely invertible with respect to an inversion $\iota$, norm-Euclidean, unit-generated, and 3-remote. Then condition (H) is satisfied.
\end{lemma}

\begin{proof}
Recall that condition (H) says that $\lim_{n\to \infty} \gamma(n)=0$, where $\gamma(n)$ is the total Lebesgue measure of rank-$n$ cylinders which are not fully contained inside any cell $F\in \mathcal{F}$.
This is at most the volume of the $\sigma(n)$-neighborhood of the boundary set $E=\cupover_{F\in \mathcal F} \partial F$, where $\sigma(n)$ is the cylinder diameter appearing in Condition (D) which we proved in Lemma \ref{lemma:condition D}. As we saw in the proof of  Proposition \ref{lemma:Proving finite range}, this boundary region naturally decomposes into a union of smooth disconnected manifolds $M_{d-1}, \ldots, M_0$ consisting of the relatively-open subsets of codimension-1 spheres and hyperplanes in dimension $k-1$ as well as the intersections of their closures in lower dimensions. By viewing each of these embedded manifolds in charts, one shows that the volume of the $\sigma(n)$-neighborhood of $E$ is bounded above by $\sum_{i=0}^{d-1} \mu_i(M_i)\sigma(n)^{k-i}$, up to multiplicative constants that do not depend on $\sigma(n)$. Since $\sigma(n)\to 0$, we also have that $\gamma(n)\to 0$, as desired.
\end{proof}

\begin{remark}
In this section, any condition that $\mathcal{Z}$ be norm-Euclidean (i.e., $\rad(K)<1$) could be replaced by the condition that $\rad(K)\le 1$.
\end{remark}

\subsection{Condition (F)}

\begin{lemma} \label{lemma:condition F}
Suppose a lattice $\mathcal{Z}$ is integral, nicely invertible with respect to an inversion $\iota$, norm-Euclidean, unit-generated, and 3-remote. Then condition (F) is satisfied.
\end{lemma}

\begin{proof}
Recall that condition (F) states that for any string $s=a_1a_2\dots a_n$ with $C_s$ nonempty, we have
\(
|\omega_s(T^n x) - \omega_s(T^n y)|\le R_1 \lambda(C_s) d(T^n x,T^n y).
\)
for all $x,y\in C_s$. 

In the proof of Lemma \ref{prop:All Uj contain full cylinders} we mentioned that all $U_j$ are, up to boundary, the intersection of half-spaces $x\cdot z<1/2$ or $x\cdot z<0$ with $z\in \mathcal{Z}$ of norm $1$ and sphere-exteriors $\{x:d(x,z)>1\}$ with $z\in\mathcal{Z}$ of norm $1$ or $\sqrt{2}$. Since $0$ is in or on the boundary of all these spaces, we have that $0\in \overline{U_j}$ for all $j$. We can extend all relevant functions to the boundary of $U_j$ continuously: in particular, $\omega_s(0)$ can be defined. Let $s'= a_1a_2\dots a_{n-1}$ and $r=T^{-1}_{s'} 0$.

Since R{\'e}nyi's condition is satisfied, we have $\omega_s(T^n z)\asymp \omega_s(0)$ for any $z\in C_{s} $.  By the definition of $\omega$ we have:
\[
\lambda(C_{s}) &= \int_{C_{s}} \ d\lambda = \int_{T^{n}C_s} \omega_s (z) \ d\lambda(z)\asymp \omega_s(0) \int_{T^{n}C_s} \ d\lambda(z) \label{eq:lambda to omega}\\& = \omega_s(0) \lambda(T^{n}C_s)\asymp \omega_s(0) \notag
\]
where the last asymptotic holds because $T^nC_{a_1\dots a_n}$ must be one of the $U_j$'s, of which there are finitely many, all of positive measure. As a result of this, we have \[\omega_s(T^n z) \asymp \lambda(C_s) \label{eq:omega to lambda full comparison}\] for any $z\in C_s$. 

By Lemma \ref{lemma:inversiondistortion}, we have that
\[\label{eq:omegadefinitioninconditionF}
\omega_s(T^n x) = \prod_{i=0}^{n-1} |T^i x|^{2d},
\]
where $d$ is the dimension of the ambient space. Since all the terms in the product are strictly less than one, we have \[\omega_{s'}(0)\ge \omega_{s}(0). \label{eq:omegas' to omegas}\]

Using the repeated inversion formula \eqref{eq:repeated inversion formula} with formula  \eqref{eq:omegadefinitioninconditionF} gives
\(
d(r,x)
&=  d(0,T^{n-1}x) \left(\prod_{i=0}^{n-2} |T^i r||T^i x| \right) =\prod_{i=0}^{n-2} |T^i r| \prod_{i=0}^{n-1}|T^i x| \\
&= \omega_{s'}(0)^{1/2d} \omega_s(T^n x)^{1/2d}.
\)
We therefore have
\[\label{eq:omegatodistanceformula}
\omega_s(T^n x) = \omega_{s'}(0)^{-1} d(r,x)^{2d}
\]
for all $x\in C_s$. Also, this gives by \eqref{eq:omega to lambda full comparison}
\[\label{eq:d(r,x)asymp}
d(r,x) \asymp \omega_{s'}(0)^{1/2d} \lambda(C_s)^{1/2d}.
\]

Finally, by using the repeated inversion formula \eqref{eq:repeated inversion formula} with \eqref{eq:omega to lambda full comparison} and \eqref{eq:omegadefinitioninconditionF}, we have
\[
d(x,y) &= \left( \prod_{i=0}^{n-1} |T^i x||T^i y| \right)d(T^n x, T^n y) \notag \\ &= \omega_s(T^n x)^{1/2d} \omega_s(T^n y)^{1/2d} d(T^n x, T^n y) \label{eq:d(x,y) to d(Tnx,Tny)}\\
&\asymp \lambda(C_s)^{1/d} d(T^n x, T^n y). \notag
\]

Combining all that we have obtained so far, we get
\(
&|\omega_s(T^nx)-\omega_s(T^n y)| \\
&\qquad = \omega_{s'}(0)^{-1}|d(r,x)^{2d}-d(r,y)^{2d}| &\text{by }\eqref{eq:omegatodistanceformula}\\
&\qquad=\omega_{s'}(0)^{-1}|d(r,x)-d(r,y)|\times \\ &\qquad \qquad \times|d(r,x)^{2d-1}+d(r,x)^{2d-2}d(r,y)+\dots+d(r,y)^{2d-1}|\\
&\qquad\asymp \omega_{s'}(0)^{-1}\max\{d(r,x)^{2d-1},d(r,y)^{2d-1}\} |d(r,x)-d(r,y)|\\
&\qquad\asymp \omega_{s'}(0)^{-1/2d}\lambda(C_s)^{(2d-1)/2d} |d(r,x)-d(r,y)| &\text{by }\eqref{eq:d(r,x)asymp}\\
&\qquad\le \omega_{s'}(0)^{-1/2d}\lambda(C_s)^{(2d-1)/2d} d(x,y)\\
&\qquad \asymp \omega_{s'}(0)^{-1/2d}\lambda(C_s)^{(2d+1)/2d} d(T^n x,T^n y) &\text{by }\eqref{eq:d(x,y) to d(Tnx,Tny)}\\
&\qquad\le \omega_s(0)^{-1/2d}\lambda(C_s)^{(2d+1)/2d} d(T^n x,T^n y) &\text{by }\eqref{eq:omegas' to omegas}\\
&\qquad\asymp \lambda(C_s) d(T^n x, T^n y) &\text{by }\eqref{eq:lambda to omega},
\)
as desired.
\end{proof}

\subsection{The 7-remote Case}\label{sec:Z3}
So far, we have assumed that our lattices are 3-remote. We now indicate how the proof changes if we instead assume 7-remoteness, which includes the case of $\Zee=\Z^3$ (vacuously as $\mathbb{Z}^3$ contains no points of norm $\sqrt{7}$).

The only proof that needs to be altered in a significant way is that of Lemma \ref{lemma:Proving finite range}. Here, we would proceed as before, but in addition to the type-1, type-2, and type-3 objects, we also require the following:
\begin{itemize}
\item (type-4) The spheres $S_1(z)$ with $z\in \mathcal{Z}$ of norm $\sqrt{3}$,
\item (type-5) The spheres $S_{1/2}(z/2)$ with $z\in \mathcal{Z}$ of norm $1, \sqrt{3}, \sqrt{5}$.
\end{itemize}
We note that in Lemma \ref{lemma:Proving finite range} we could exclude type-4 objects by the 3-remote condition, but now they must be considered.

Before proceeding, let us consider how inversion acts on spheres more carefully. Consider $S_r(z)$ with $|z|>r$. The point on the sphere nearest $0$ is $\frac{|z|-r}{|z|}z$ and the point on the sphere farthest from $0$ is $\frac{|z|+r}{|z|}z$. Therefore the point on  $\iota S_r(z)$  nearest $0$ at $\frac{|z|}{|z|+r}\iota(z)= \frac{1}{|z|+r} \frac{\iota(z)}{|\iota(z)|}$ and the point on $\iota S_r(z)$ farthest from $0$ is at $\frac{|z|}{|z|-r}\iota(z)= \frac{1}{|z|-r} \frac{\iota(z)}{|\iota(z)|}$. From this we can quickly calculate that the new sphere is
\(
\iota S_r(z) = S_{\frac{r}{|z|^2-r^2}}\left( \frac{1}{|z|^2-r^2} |z|^2 \iota (z) \right),
\)
where the center of this new sphere has norm $|z|/(|z|^2-r^2)$. Note that by our assumption of nice invertibility, if $z\in \mathcal{Z}$, then $|z|^2 \iota(z)\in \mathcal{Z}$.

Thus, if $S_1(z)$ is a type-4 object, then $\iota(S_1(z))$ is $S_{1/2}(z'/2)$ for some $z'\in \mathcal{Z}$ with $|z'|=\sqrt{3}$. As before, we can translate this new sphere to see where it may intersect $K$. The resulting center will be at $z'/2+w$ for some $w\in \mathcal{Z}$. Knowing that $|z'|^2=3$, we have
\(
\left| \frac{z'}{2}+w\right|^2 = \frac{3}{4}+z'\cdot w+w\cdot w.
\)
Since $2z'\cdot w \in \mathbb{Z}$, we have $|z'/2+w|^2$ must be an odd integer over $4$. Let us consider the options. Suppose $|z'/2+w|^2\ge 9/4$. Then, the sphere $S_{1/2}(z'/2+w)$ cannot intersect the open unit ball, and therefore does not intersect $K$, by the norm-Euclidean condition.  Suppose next that $|z'/2+w|^2=7/4$ and that $S_{1/2}(z'/2+w)$ intersects $K$ in a nontrivial way. We then have a point $z'+2w\in \Zee$ of norm $\sqrt{7}$, such that $d(z'/2+w, K)<1/2$. Our assumption of 7-remoteness rules out this option. Therefore, $\iota S_1(z)+w=S_{1/2}(z'/2+w)$ is a type-5 object.

So consider now type-5 objects. 

First, if $|z|=\sqrt{3}$, then $\iota S_{1/2}(z/2)$ will be $S_1(z')$ for $z'\in\mathcal{Z}$. Its translates that intersect $K\subset B_1(0)$ are therefore necessarily type-3 or type-4 objects. 

If $|z|=\sqrt{5}$, then $\iota S_{1/2}(z/2)$ will be $S_{1/2}(z'/2)$ for some $z'\in\mathcal{Z}$. We showed above that translates of such spheres that intersect $K$ are all type-5 objects.

Finally, we have the spheres $S_{1/2}(z/2)$ where $|z|=1$. This sphere contains the origin, and its farthest point from the origin is at $z$. So therefore $\iota S_{1/2}(z/2)$ will be the hyperplane $x\cdot \iota(z) =1$. The translates of these hyperplanes that intersect $K$ will be type-1 or type-2 objects.

This completes the proof of the altered version of Lemma \ref{lemma:Proving finite range}.

Likewise the proof of Lemma \ref{prop:All Uj contain full cylinders} is functionally unchanged.

\section{Proving the Measure is Real-analytic}\label{sec:real analytic}

In this section, we extend an argument of Hensley \cite{HensleyBook}, in turn based on Bandtlow--Jenkinson \cite{bandtlow2007invariant} and Mayer \cite{MR757051} to study the invariant measure for $T$. Adjusting some techniques and filling in some details, we prove:


\begin{thm}\label{thm:real-analytic}
Under the hypotheses of Theorem \ref{thm:main}, the invariant measure for $T$ has a density $d\mu/d\lambda$ that is analytic on each component of $K\setminus \cupover_n T^n\partial K$.
\begin{proof}
We follow Hensley's method, relegating calculations to later lemmas.


First, recall the notation for the serendipitous decomposition of $K$. The proof of Lemma \ref{lemma:Proving finite range} and its analogue in Section \ref{sec:Z3} show that the forward orbit of the boundary $E=\cupover_{n=0}^\infty T^n \partial K$ can be written as a finite union of hyperplanar and spherical objects, cutting $K$ into a finite number of open connected components $B_i$ with indexing set $\mathcal{I}$, so that $K=E\cup \bigcup_{i\in\mathcal{I}} B_i$. For each $i,j$, let $G_{i,j}=\{a\in \Zee: T_a^{-1}(B_j)\cap B_i\neq \emptyset\}$. Note that for a fixed $j$, the sets $G_{i,j}$ are pairwise disjoint. From the fact that $E$ already contains all images of $\partial K$, it follows immediately (Lemma \ref{lem:Gij containment}) that for each $a\in G_{i,j}$, we in fact have $T_a^{-1}(B_j)\subset B_i$.

We next rephrase the theorem as an eigenvalue problem. Namely, consider the space $\mathcal B_\R$ of bounded, locally-analytic functions on $K\setminus E$. Viewing functions in $\mathcal B_\R$ as densities for finite measures, we can define a transfer operator $L_\R$ which maps a density $f$ to the push-forward density \(
L_\R f  (x)=\sum_{i\in\mathcal{I}} \sum_{a\in G_{i,j}} f(T_a^{-1}(x)) w_a(x), \qquad \text{if }x\in B_j
\) where $w_a(x)=\norm{a+x}^{-2d}$ is the Jacobian of $T_a^{-1}$ at $x$.

We will want to prove two things: that there is a subspace $\mathcal B'_\R\subset \mathcal B_\R$ such that $L_\R$ defines a well-defined operator $L_\R:\mathcal B'_\R\rightarrow \mathcal B'_\R$, and that for an appropriate norm on $\mathcal B'_\R$ the transfer operator is compact. We will then apply the theory of positive operators to find a 1-eigenfunction for $\mathcal B'_\R$.

We will need to construct several intermediate spaces, which will inter-relate as follows:
$$\prod_{i\in \mathcal I} H^\infty D_i \supset \prod_{i\in \mathcal I} \mathcal B_i \stackrel{I}{\longrightarrow} I\left(\prod_{i\in \mathcal I} \mathcal B_i\right)=\mathcal B'_\R \subset \mathcal B_\R.$$

We now complexify all objects involved, starting with extending $\R^d$ to $\C^d$. Recall that the inversion $\iota$ is given by a composition of an orthogonal linear transformation $\mathcal O$ with the mapping $x\mapsto x/\norm{x}^2$. For $z=(z_1, \ldots, z_d)\in \C^d$, take $Q(z)=\sum z_i^2$, extend $\iota$ to complex coordinates as $\iota(z)=\mathcal O(z)/Q(z)$, and define $T^{-1}_a(z) = \iota(a+z)$ (note that we do not extend $T$ to $\C^d$). We also extend the Jacobian to be $w_a(z) = Q(z)^{-d}$. This is no longer the Jacobian of $T^{-1}_a(z)$, except on the real part $\R^d$.

\begin{figure}
    \centering
    \includegraphics[width=.65\textwidth]{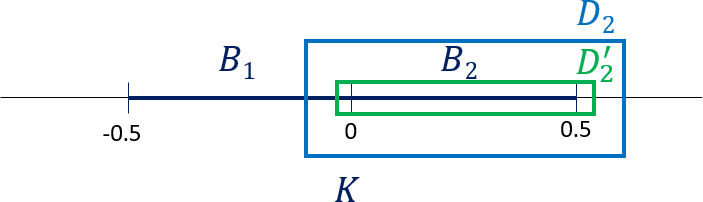}
    \caption{Thickening the serendipitous decomposition of the nearest-integer CF system with inversion $x\mapsto 1/x$. Here, $E=\{-0.5,0,0.5\}$ breaks $K=(-0.5,0.5]$ into two regions $B_1$ and $B_2$. When $B_2$ is thickened in $\C$ to create $D_2$, its image under $T_a^{-1}$, for $a\in G_{2,2}=\{a\geq 2\}$, is contained in a smaller region $D'_j$.}
    \label{fig:thickening}
\end{figure}

Let $U\subset \C^d$ be the unit ball, and $\epsilon<1$ to be determined later. 
For $a\in \Zee\setminus\{0\}$ and sufficiently small $\epsilon$, we have that $T_a^{-1}$ is defined and contractive on $\R^d +\epsilon U$ away from the origin (Lemma \ref{lem:contraction principle}), sending each set $D_i=B_i+\epsilon U$ into a set $D'_j = B_j + \delta U \subset D_j$ with $\delta<\epsilon$ (with $\delta$ not depending on $a$).

Thickening our function spaces, for each index $i\in \mathcal I$ let $H^\infty D_i$ be the space of bounded holomorphic functions on $D_i$ with the sup norm, and define $H^\infty D'_i$ likewise. Define the complexified transfer operator $L_{i,j}: H^\infty D_i \rightarrow H^\infty D_j$ as
\(
L_{i,j}f(z)= \sum_{a\in G_{i,j}} f(T_a^{-1}(z)) w_a(z).
\)

We next prove that $H^\infty D_i \stackrel{L_{i,j}}{\rightarrow} H^\infty D_j$ is well-defined and is a compact operator. Indeed, by definition of $D'_i$, the operator $H^\infty D_i \stackrel{L_{i,j}}{\rightarrow} H^\infty D_j$ only makes use of the values of $f\vert_{D'_i}$, so it factors as 
$$H^\infty D_i \stackrel{J}{\rightarrow} H^\infty D'_i \stackrel{L_{i,j}\vert_{H^\infty D'_i}}{\longrightarrow}  H^\infty D_j,$$ as the composition of the canonical embedding $J:H^\infty D_i \rightarrow H^\infty D'_i$ that views holomorphic functions on $D_i$ as holomorphic functions on $D'_i$ (compact since, by Montel's Theorem, bounded sequences of holomorphic functions sub-converge \emph{on compacts} and $D'_i$ has compact closure in $D_i$) and the restriction $L_{i,j}\vert_{H^\infty D'_i}$.  We show (Lemma \ref{lem:Sij bound}) that $L_{i,j}(1)$ is a bounded function, and then use the Vitali convergence theorem to prove (Lemma \ref{lem:Lij bounded}) that  $L_{i,j}\vert_{H^\infty D'_i}$ is bounded. 

Next, let $\mathcal B_i=\{f\in H^\infty D_i: f(B_i)\subset \R\}\subset H^\infty D_i$, a closed subspace. Observing that $L_{i,j}(\mathcal B_i) \subset \mathcal B_j$, we restrict $L_{i,j}: H^\infty D_i\rightarrow H^\infty D_j$ to $L_{i,j}: \mathcal B_i \rightarrow \mathcal B_j$, which remains a compact operator.

Combining this construction for all indices $i$, we take $\mathcal{B}_\C=\prod_{i\in\mathcal{I}} \mathcal B_i$, and define the transfer operator $L_\C$ as a matrix of transfer operators $L_{i,j}$. 
Namely, if $f=(f_i)_{i\in\mathcal{I}}\in \mathcal B_i$, then 
\(
L_\C f = \left( \sum_{i\in\mathcal{I}} L_{i,j}f_i\right)_{j\in\mathcal{I}}.
\)

Returning to real coordinates, define a mapping $I$ from $\mathcal B_\C$ into $\mathcal B_\R$ by restricting each multifunction $(f_i)\in \mathcal B_\C$ to $B_i$, jointly giving a piecewise-defined locally-analytic function on $K\setminus E$. Let $\mathcal B'_\R$ be the image of this operator. Observe that $I$ is linear and injective: if $f\in \ker I$ is given by $(f_i)$, then each $f_i$ must be identically zero on the open set $B_i$, so that all derivatives must be zero along $B_i$, and the power series expansion must be identically zero on a neighborhood of $B_i$, which would then imply by the identity theorem that $f_i$ is identically zero on $D_i$. Thus, $I: \mathcal B_\C \rightarrow \mathcal B'_\R$ is an isomorphism of vector spaces. Give $\mathcal B'_\R$ the induced norm: namely, for $f\in \mathcal B'_\R$, $\Norm{f}=\Norm{I^{-1}(f)}_\infty$ is the sup norm of the piecewise-extension of $f$ to the complexified regions $D_i$. 
Finally, observe that $I$ provides a conjugacy between $L_\R$ and $L_\C$, so that $L_\R: \mathcal B'_\R \rightarrow \mathcal B'_\R$ is a compact operator.

Next, we apply the theory of positive operators to obtain a unique eigenvalue for $L_\R$, and exponential convergence to this eigenvalue for all positive densities. 
To this end, let $P\subset \mathcal \mathcal B'_\R$ be the subset of non-negative functions.

The following are clear:  (1) $P$ is closed under addition and scaling by positive numbers, (2) the interior of $P$ is non-empty (in particular small perturbations of the function $f(x)=1$ remain in $P$), (3) $P$ is closed, (4) any element of $\mathcal B'_\R$ is a difference of two elements of $P$ (note that  $f$  is bounded above), (5) $L_\R$ maps $P$ into $P$. 

To apply the theory, it remains to show (6) that for any non-zero $f\in P$, there is a $n$ such that $L^n_\R(f)\asymp 1$. Since $f$ is a continuous function on $K\setminus E$, it must be positive on some open set, which in turn contains a full cylinder for $T$ at some depth $n$ by Lemma \ref{lem:full cylinders are dense}. Thus, $n$ iterations of the transfer operator extend this region of positivity to all of $K$, while also adding some other non-negative values, so that we have a uniform bound  $0<\inf_{x\in K} L^n_\R(f)(x)$. Since $L^n_\R f\in \mathcal B'_\R$, it is the restriction of a bounded multi-function in $\mathcal B_\C$, and therefore is bounded uniformly above by $\Norm{f}_{B'_\R}$, the supremum of the extended multi-function.

The above conditions then imply, by  Theorem 2.5 of \cite{MR0181881}, that $L_\R$ has a positive eigenvalue $\rho$ with eigenfunction $f\in P$. Furthermore, since $L_\R$ is a transfer operator and preserves the $L^1$ norm, we have $\Norm{f}_1=\Norm{L_\R f}_1 = \Norm{\rho f }_1 =\rho \Norm{f}_1$ and so $\rho=1$. By definition of the space $\mathcal B'_\R$, $f$ is locally-analytic on $K\setminus E$ (and furthermore can be extended complex-analytically from each $B_i$ to $D_i$). 

Treating the eigendirection for the transfer operator as a probability density function, we obtain a $T$-invariant measure on $K$. Since Theorem \ref{thm:main} already provided a \emph{unique} invariant measure $\mu$ equivalent to Lebesgue measure, we conclude that $\mu$ is, in fact, given by our density, which is by construction locally-analytic on $K\setminus E$.

Lastly, we remark that parts of Theorem \ref{thm:main}, including exactness, can also be obtained by using a denseness argument, as described in \cite{bandtlow2007invariant}.
\end{proof}
\end{thm}

\subsection{Technical Lemmas}

Here we prove the lemmas used to prove Theorem \ref{thm:real-analytic}.

\begin{lemma}\label{lem:Gij containment}
For any $i,j\in \mathcal{I}$ and any $a\in \Zee$, we have that either $T_a^{-1}(B_j)\subset B_i$ or $T_a^{-1}(B_j)\subset B_i^c$. 
\end{lemma}

\begin{proof}
Observe first that, by construction, $T_a^{-1}$ maps $B_j$ to the complement of $E$: otherwise, we would have a point of $E$ whose image under $T$ is in $B_j\subset K\setminus E$, but $T$ maps points of $E$ to $E$, a contradiction. Since $T_a^{-1}$ is continuous and $B_i, B_j$ are both connected components of $E$, this implies that $T_a^{-1}(B_j)\cap B_i\neq \emptyset$ implies $T_a^{-1}(B_j)\subset B_i$, as desired.\qedhere


\end{proof}

\begin{lemma}
\label{lemma:complexisreal}
Let $b\in \R^d$ satisfying $\norm{b}>1$. Let $0<\epsilon<1$ and $c\in \C^d$ satisfying $\norm{c}<\epsilon$. For sufficiently small values of  $\epsilon/\norm{b}$ one has
\[\label{eq:nearness bound}
\frac{1}{\sum_{i=1}^d (b_i+c_i)^2} = 
\frac{1}{\sum_{i=1}^d b_i^2} \left(1+O\left( \frac{\epsilon}{|b|}\right)\right)
\]
with uniform implicit constant.
\end{lemma}

\begin{proof}
We have the following, applying the Cauchy-Schwartz inequality in the second step,  
\(
\frac{\sum_{i=1}^d (b_i +c_i)^2}{\sum_{i=1}^d b_i^2} &= 1 + \frac{2\sum_{i=1}^d b_i c_i}{|b|^2} + \frac{\sum_{i=1}^d c_i^2}{|b|^2}= 1+O\left( \frac{2|b||c|}{|b|^2} \right) +O\left( \frac{|c|^2}{|b|^2} \right)\\
&= 1+ O\left( \frac{2\epsilon }{|b|} \right) +O\left( \frac{\epsilon^2}{|b|^2} \right)= 1+O\left( \frac{\epsilon}{|b|}\right).
\)
Thus
\(
\frac{1}{\sum_{i=1}^d (b_i+c_i)^2} = 
\frac{1}{\left(\sum_{i=1}^d b_i^2\right)\left(1+O\left( \frac{\epsilon}{|b|}\right)\right)}. 
\)
The result then follows by noting that $(1+f(x))^{-1}=1+O(f(x))$ with uniform implicit constant when $|f(x)|$ is both less than and bounded away from $1$.
\end{proof}

The following lemma generalizes Lemma 5.2 of \cite{HensleyBook}, with an alternate proof.
\begin{lemma}\label{lem:contraction principle}
For sufficiently small $\epsilon$, there exists $0<\delta<\epsilon$ such that for all $a\in G_{i,j}$, we have $T_a^{-1}(B_j+\epsilon U) \subset B_i+ \delta U$, where $U$ is the unit ball in $\C^d$.
\end{lemma}

\begin{proof}
Let $x_0\in B_j$, $a\in G_{i,j}$, and $z\in \epsilon U$. Write $x=a+x_0$, so that $T_a^{-1}(x_0+z) = \iota (x+z)$ and $\norm{x}>(\rad K)^{-1}>1$. The inversion $\iota$ is the composition of the mapping $x\mapsto x/\norm{x}^2$ with some orthogonal mapping $\mathcal O$. 

If $U$ were a subset of $\R^d$, the result would be immediate from properness of $K$ and the inversion identity \eqref{eq:inversion formula}, as long as $\epsilon<(\rad(K)^{-1}-1)$. We reduce to this case in two ways, for large and small choices of $a$, respectively.

For small digits, observe that \eqref{eq:inversion formula} implies that the singular values of the differential $D\iota$ (viewed either as a real or complex mapping along $\R^d$) are strictly smaller than $1$ at points that are strictly outside of the unit ball. By continuity of the differential, the complex-analytic extension of $\iota$ remains a contraction on a neighborhood of any point, and thus of any compact region that $x=a+x_0$ may lie in. This gives a uniform estimate for any finite set of digits $a$, but we don't have control over the \emph{full} collection $G_{i,j}$.

For sufficiently large digits $a$, we may use Lemma \ref{lemma:complexisreal} to adjust the denominator in the inversion:
\begin{align*}
    \frac{\mathcal{O}(x+z)}{Q(x+z)} &= \frac{\mathcal{O}(x+z)}{\norm{x}^2}\left(1+O\left(\frac{\epsilon}{\norm{x}}\right)\right)\\ &= \frac{\mathcal{O}(x)}{\norm{x}^2}
       + \frac{\mathcal{O}(x)}{\norm{x}^2}\cdot  O\left(\frac{\epsilon}{\norm{x}}\right)+\frac{\mathcal{O}(z)}{\norm{x}^2}\left(1+O\left(\frac{\epsilon}{\norm{x}} \right)\right).
\end{align*}

The first term is in $B_i$, as desired. It remains to bound the remaining terms uniformly for sufficiently large $x$, corresponding to large $a$:
\begin{align*}
    &\norm{\frac{\mathcal{O}(x)}{\norm{x}^2}\cdot  O\left(\frac{\epsilon}{\norm{x}} \right)+\frac{\mathcal{O}(z)}{\norm{x}^2}\left(1+O\left(\frac{\epsilon}{\norm{x}}\right)\right)}\\ &\qquad \leq \norm{\frac{\mathcal{O}(x)}{\norm{x}^2}\cdot  O\left(\frac{\epsilon }{\norm{x}}\right)}+\norm{\frac{\mathcal{O}(z)}{\norm{x}^2}\left(1+O\left(\frac{\epsilon}{\norm{x}}\right)\right)}
    \\&\qquad\leq \norm{O\left(\frac{\epsilon}{\norm{x}^2}\right)}+\norm{\frac{\epsilon}{\norm{x}^2}\left(1+O\left(\frac{\epsilon}{\norm{x}}\right)\right)}= O\left(\frac{\epsilon}{\norm{x}^2}\right)<\delta<\epsilon,\end{align*}
    for some $\delta$ not depending on $x$, for sufficiently large $x$.
\end{proof}

\begin{lemma}\label{lem:Sij bound}
For any $i,j\in \mathcal{I}$, we have that 
\[
S_{i,j}:=\sup_{z\in D_j} \sum_{a\in G_{i,j}} \left|w_a(z)\right|<\infty.
\]
\end{lemma}

\begin{proof}
Let $z\in D_j$ and $a\in G_{i,j}$. Assume, for the moment that $|a|$ is large, say greater than some large positive number $R$. Then, applying a variant of \eqref{eq:nearness bound} with $b=a$ and $c=z$ gives
\(
w_a(z)&=\left(\sum_{i=1}^d(a_i+z_i)^2\right)^{-d} = \left(\left(\sum_{i=1}^d a_i^2\right) (1+O(|a|^{-1}))\right)^{-d}\\
&= |a|^{-2d} (1+O(|a|^{-1})) = O\left(|a|^{-2d}\right).
\)
The implicit constant in the big-O notation is uniform over all large $a$'s, which allows us to write the following:
\(
\sum_{|a|>R} w_a(z) = \sum_{|a|>R} O\left(|a|^{-2d}\right) = O\left( \sum_{|a|>R} |a|^{-2d}\right).
\)
 Now we consider an estimate on the number of terms $a$ where $|a|^2 = n$ (since $\Zee$ is integral, these are the only possibilities). Consider the annular region $\{x\in\R^d: \sqrt{n}-\rad(K)\le |x|\le\sqrt{n}+\rad(K)\}$, which has volume $O(\sqrt{n}^{d-1})$, since it's a thickening of the sphere of radius $\sqrt{n}$ in $\R^{d-1}$. On the other hand, for each such $a$, the shifted Dirichlet domain $a+K$ must be contained in the annular region, so the volume of the annular region is bounded below by $\norm{\{a\in \Zee: \norm{a}^2=n\} }\operatorname{vol}(K)$. Combining these, we obtain the estimate $\norm{\{a\in \Zee: \norm{a}^2=n\} }= O(n^{(d-1)/2})$.
From here, we obtain
\(
\sum_{|a|>R} w_a(z)  &= O\left( \sum_{|a|>R} |a|^{-2d}\right)= O\left( \sum_{n>R^2} \frac{n^{(d-1)/2}}{n^d} \right) 
&=  O\left( \sum_{n>R^2} \frac{1}{n^{(d+1)/2}} \right),
\)
which, if the dimension $d$ satisfies $d>1$, is convergent and uniformly bounded. When $d=1$, we instead analyze $ O\left( \sum_{|a|>R} |a|^{-2}\right)$ by using the fact that $\Zee$ is a lattice, so that $|a|$ will belong to an arithmetic progression, and $\sum_{|a|>R} |a|^{-2}$ converges like the tail of $\sum_{m>0} m^{-2}$.

For the remaining $a$'s with $|a|\le R$, we let $x$ be the nearest point to $z$ that lies in $B_j$. So $d(x,z)<\epsilon$. We then use \eqref{eq:nearness bound} with $b=a+x$ and $c=z-x$. Then we have
\(
\sum_{|a|\le R} w_a(z)&=\sum_{|a|\le R} \left(\frac{1}{\sum_{i=1}^d (a_i+z_i)^2}\right)^{d}\\ &= \left( \sum_{|a|\le R} \left(\frac{1}{\sum_{i=1}^d (a_i+x_i)^2}\right)^{d}\right)(1+O(\epsilon)).
\)
Since $|a+x|\ge \rad(K)^{-1}$, this is now a finite sum of bounded terms and thus is bounded, as desired.
\end{proof}

\begin{lemma}\label{lem:Lij bounded}
$L_{i,j}$ is a bounded linear operator from $H^\infty D_i'$ into $H^\infty D_j$.
\end{lemma}

\begin{proof}
Fix $f\in H^\infty D_i'$, and for $R\in \mathbb{N}$, consider the bounded-digit sum
\(
g_R(z)=\sum_{|a|\le R, a\in G_{i,j}} w_a(z) f(T_a^{-1} z).
\)
Then we have that
\(
|g_R(z)| &\le \sum_{|a|\le R, a\in G_{i,j}} |w_a(z)|\cdot \left| f(T_a^{-1} z)\right|\\ &\le \left( \sum_{|a|\le R, a\in G_{i,j}} |w_a(z)|\right) \cdot \sup_{a\in G_{i,j}} \left| f(T_a^{-1} z) \right| \\
&\le S_{i,j} \cdot \Norm{f}_{H^{\infty}D_i'},
\)
where $S_{i,j}$ was as defined in Lemma \ref{lem:Sij bound}. Thus, the sequence $\{g_R\}$ is uniformly bounded on $D_j$. By appealing to the proof of Lemma \ref{lem:Sij bound} as needed, we can moreover show that for any point $z\in D_j$, the sequence $\{g_R(z)\}$ is Cauchy and therefore  $\lim_{R\to \infty} g_R(z)$ converges to something we will call $g(z)$.  By Vitali's convergence theorem \cite[Prop.~7]{narasimhan1971several}, $g_R$ converges uniformly to $g$ on compact subsets of $D_j$ and so $g$ is analytic on all of $D_j$. Morever, since $|g_R(z)|\le S_{i,j}\cdot \Norm{f}_{H^{\infty} D_i'}$, we also have $|g(z)|\le S_{i,j}\cdot \Norm{f}_{H^{\infty} D_i'}$. Thus $g=L_{i,j}f$ and $\Norm{L_{i,j}f}_{H^{\infty}D_j} \le S_{i,j}\cdot \Norm{f}_{H^{\infty} D_i'}$, which completes the proof.
\end{proof}

\begin{lemma}\label{lem:full cylinders are dense}
Every open disk inside every $B_i$, $i\in \mathcal{I}$, contains a full cylinder.
\end{lemma}

\begin{proof}
Since condition (D) is satisfied (Lemma \ref{lemma:condition D}), we know that any such open disk must contain a cylinder $C_{a_1\dots a_n}$ (not necessarily a full cylinder). Moreover, $T^n C_{a_1\dots a_n} = U_j$ for some $U_j$, and $U_j$ contains a full cylinder $C_{b_1\dots b_m}$ by condition (E) (Lemma \ref{prop:All Uj contain full cylinders}). Thus $C_{a_1\dots a_n b_1\dots b_m}$ is contained in our open disk and is full, as desired.
\end{proof}

\bibliographystyle{plain}
\bibliography{bibliography}

\end{document}